\documentclass[12pt]{article}
\usepackage[margin = 20mm]{geometry}
\usepackage{amsmath}
\usepackage{amsfonts}
\usepackage{mathrsfs}
\usepackage{amssymb}
\usepackage{amsthm}
\usepackage{cases}
\usepackage{enumerate}
\usepackage{tikz}
\usepackage{pgfplots}
\pgfplotsset{compat=1.5}
\usepackage{tocloft}
\usepackage{setspace}
\usepackage{abstract}
\usepackage{nicematrix}
\usepackage{xcolor, colortbl}
\usepackage[ruled]{algorithm2e}
\usepackage[backend = biber, doi = false, url = false, isbn = false]{biblatex}
\renewbibmacro{in:}{}
\newtheorem{thm}{Theorem}[section]
\newtheorem{cor}[thm]{Corollary}
\newtheorem{lem}[thm]{Lemma}

\newtheorem{conj}[thm]{Conjecture}

\theoremstyle{definition}
\newtheorem{defin}[thm]{Definition}

\newtheorem{ex}[thm]{Example}
\numberwithin{equation}{section}
\def\eref#1{$(\ref{#1})$}
\def\sref#1{\S$\ref{#1}$}
\def\lref#1{Lemma~$\ref{#1}$}
\def\tref#1{Theorem~$\ref{#1}$}
\def\cyref#1{Corollary~$\ref{#1}$}

\def\dref#1{Definition~$\ref{#1}$}

\def\F{\mathcal{F}}
\def\C{\mathcal{C}}
\def\V{\mathscr{V}}

\def\Z{\mathbb{Z}}
\def\syms{\Z_p}
\def\FF{\mathbb{F}_q}
\def\lcm{\text{lcm}}
\def\M{\mathscr{A}}
\def\N{\mathscr{B}}
\def\Np{\mathscr{N}}
\def\Rp{\mathscr{R}}

\renewcommand{\geq}{\geqslant}
\renewcommand{\leq}{\leqslant}
\renewcommand{\ge}{\geqslant}
\renewcommand{\le}{\leqslant}
\renewcommand{\emptyset}{\varnothing}
\makeatletter
\g@addto@macro\bfseries{\boldmath}
\makeatother

\title{Row-Hamiltonian Latin squares and Falconer varieties
}
\author{Jack Allsop \ \ Ian M. Wanless\\
  \small School of Mathematics\\[-0.5ex]
  \small Monash University\\[-0.5ex]
  \small Vic 3800, Australia\\
  \small\tt jack.allsop@monash.edu \ \ ian.wanless@monash.edu}

\date{}

\begin{document}

\maketitle

\begin{abstract}
  A \emph{Latin square} is a matrix of symbols such that each symbol
  occurs exactly once in each row and column. A Latin square
  $L$ is \emph{row-Hamiltonian} if the permutation induced by each pair
  of distinct rows of $L$ is a full cycle permutation. Row-Hamiltonian
  Latin squares are equivalent to perfect $1$-factorisations of
  complete bipartite graphs.
  For the first time, we exhibit a family of Latin squares that are
  row-Hamiltonian and also achieve precisely one of the related
  properties of being column-Hamiltonian or symbol-Hamiltonian.  This
  family allows us to construct non-trivial, anti-associative,
  isotopically $L$-closed loop varieties, solving an open problem
  posed by Falconer in 1970.
  
  \medskip
  

\end{abstract}

\section{Introduction}\label{s:intro}

A \emph{quasigroup} $(Q,*)$ is a non-empty set $Q$ with a binary
operation $*$ such that, for each $a,b\in Q$, there exist unique
$x,y\in Q$ for which $a*x=b$ and $y*a=b$. The operation table for $*$
is a \emph{Latin square}, meaning that each element of $Q$ occurs
precisely once in each row and column. Conversely, when the rows and
columns of any Latin square are indexed by the symbols in the square,
the result is the \emph{Cayley table} of a quasigroup.  We will be
concerned primarily with finite quasigroups and Latin squares, and will
treat them as interchangeable notions.

It is often helpful to consider a quasigroup $(Q,*)$ in its equivalent
form as a set of \emph{triples} $(x,y,x*y)\in Q^3$. In this viewpoint,
the group $\text{Sym}(Q)^3$ has a natural action on the quasigroups
with underlying set $Q$, where $\text{Sym}(Q)$ is the set of
permutations of $Q$. The orbit of $(Q,*)$ under this action is the set
of \emph{isotopes} of $(Q,*)$ and the quasigroups in it are said to be
\emph{isotopic} to $(Q,*)$. The orbit of $(Q,*)$ under the action of
the diagonal subgroup of $\text{Sym}(Q)^3$ is the set of quasigroups
\emph{isomorphic} to $(Q,*)$. There is also a natural action of
$\text{Sym}(\{1,2,3\})$ which uniformly permutes the triples of a
quasigroup to obtain one of six \emph{conjugate} quasigroups (also called 
\emph{parastrophes}).

Suppose that $i$ and $j$ are two distinct rows of a Latin square $L$
with symbol set $S$. We define a permutation $r_{i,j}: S \to S$ by
$r_{i,j}(L_{i,k}) = L_{j,k}$ for each $k \in S$. We call $r_{i,j}$ a
{\em row permutation} between rows $i$ and $j$. Writing $r_{i,j}$ in
disjoint cycle notation, each cycle that we get is a \emph{row cycle}.
We say that a Latin square is {\it row-Hamiltonian\/} if every row
permutation consists of a single cycle. Isotopy preserves the length
of row cycles, and hence preserves the row-Hamiltonian
property. Row-Hamiltonian Latin squares are highly prized, in part
because of their close relationship to perfect $1$-factorisations of
graphs, as we explain in \sref{s:bm}. Let $\nu=\nu(L)$ denote the
number of conjugates of $L$ that are row-Hamiltonian. Our first
main result is to construct a family of Latin squares with $\nu=4$.
No such family was previously known.

\begin{thm}\label{t:RCnotS}
  Let $p > 3$ be a prime such that $p\equiv1\bmod 8$ or $p\equiv3\bmod
  8$. Then there is a Latin square $\mathscr{L}_p$ of order $p$ with
  $\nu(\mathscr{L}_p)=4$.
\end{thm}

In proving \tref{t:RCnotS}, we will settle a conjecture first
published in~\cite{MR2563050}. The key to proving the conjecture is to
convert it into a problem in linear algebra. Our approach is novel
for row-Hamiltonian Latin squares, which had previously only been studied
using combinatorial methods.

\medskip

The second main result of this paper is a solution to an open problem
due to Falconer~\cite{MR272932}. Before we can state our result we
must give some further definitions. 
A quasigroup $(Q,*)$ is a \emph{loop} if there is some
$e \in Q$ such that $e * q = q * e = q$ for all $q \in Q$. Every
quasigroup is isotopic to at least one loop.  Given a set $E$ of
quasigroup identities, we call the set of quasigroups satisfying all
identities in $E$ the \emph{quasigroup variety} defined by
$E$. Similarly, the set of all loops satisfying all identities
in $E$ is the \emph{loop variety} defined by $E$. For example, the loop
variety defined by the associative identity is simply the set of all
groups. Many quasigroup varieties have been extensively studied. For
an introduction to the subject, see Chapter 2 of~\cite{MR3495977}.

We call a loop variety \emph{isotopically $L$-closed} if it contains every loop that is isotopic to any element of the variety. A variety is \emph{anti-associative} if the only group it contains is the trivial group. The trivial group is a member of any variety, as it satisfies any quasigroup identity. Any other member of the variety is \emph{non-trivial} and we call the variety \emph{non-trivial} if it has a non-trivial member.

We are now in a position to state Falconer's problem. She asked if
there exists a non-trivial, anti-associative, isotopically $L$-closed
loop variety. In Falconer's honour we call such loop varieties
\emph{Falconer varieties}. We answer her problem by proving the
existence of infinitely many Falconer varieties.

\begin{thm}\label{t:falcsol}
  For each prime $p>3$ with $p\equiv1\bmod8$ or $p\equiv3\bmod8$,
  there exists a Falconer variety $\V_p$ whose smallest non-trivial
  member is a loop of order $p$.
\end{thm}

The structure of the paper is as follows. In \sref{s:bm} we will
demonstrate the motivation behind \tref{t:RCnotS}, and also give
further background information. In \sref{s:main} we will define an
infinite family of Latin squares, and show that this family has
$\nu=4$, proving \tref{t:RCnotS}. In \sref{s:app} we will use the
Latin squares constructed in \sref{s:main} to prove
\tref{t:falcsol}. In \sref{s:c} we give some directions for future
research.

\section{Motivation and background}\label{s:bm}

A \emph{$1$-factor}, or \emph{perfect matching}, of a graph $G$ is a
subset $M$ of the edge set of $G$ such that every vertex is incident
to exactly one edge in $M$. A \emph{$1$-factorisation} of a graph $G$
is a partition of the edge set of $G$ into $1$-factors. We call a
$1$-factorisation \emph{perfect} if the union of any two distinct
$1$-factors in the $1$-factorisation forms a Hamiltonian cycle in $G$.
Kotzig~\cite{MR0173249} famously conjectured that there is a perfect
$1$-factorisation of the complete graph $K_{2n}$ for any positive
integer $n$. Much effort has been put into attacking this conjecture
but it is still very far from being resolved. Historically, each
increase in the smallest order for which Kotzig's conjecture is
unresolved has merited a paper.  See~\cite{MR3954017} for a discussion
of the history and~\cite{MR4070707} for the complete list of orders
for which Kotzig's conjecture is known to hold.

Row-Hamiltonian Latin squares are equivalent to
perfect $1$-factorisations of complete bipartite graphs.
Given $\F$, a perfect $1$-factorisation of $K_{n}$, you can construct a
perfect $1$-factorisation of the complete bipartite graph
$K_{n-1,n-1}$. Indeed, from $\F$ you can obtain a number of non-isomorphic
perfect $1$-factorisations of $K_{n-1,n-1}$ equal to the number of
orbits of the automorphism group of $\F$.
For further details of all these connections, see~\cite{MR2130738}.
Not all perfect $1$-factorisations
of complete bipartite graphs are obtainable from perfect $1$-factorisations
of complete graphs. This is important for our purposes, because the
perfect $1$-factorisations that are obtained in this way create
Latin squares with $\nu\in\{2,6\}$, and hence are no use for proving
\tref{t:RCnotS}.

Another major reason for interest in row-Hamiltonian Latin squares is
their relationship to the class of Latin squares which are devoid of
proper Latin subsquares. Such squares, called $N_\infty$ squares, are
conjectured to exist for all orders $n\notin\{4,6\}$. This has been
proved~\cite{MR772581, MR2261822} for all orders not of the form
$2^a3^b$ for $a\ge1$ and $b\ge0$.  It was noted in~\cite{MR1670298}
that a Latin square of order $n$ is row-Hamiltonian if and only if it
contains no $a \times b$ Latin subrectangle with $1 < a \leq b < n$.
This is a much stronger property than $N_\infty$, which itself is
extremely rarely achieved. It is known (see, e.g.~\cite{arXiv:2106.11932})
that the overwhelming majority of Latin squares have quadratically many
Latin subsquares.

As mentioned in \sref{s:intro}, each Latin square has six
\emph{conjugate} squares obtained by uniformly permuting the
co-ordinates within its triples.  These conjugates can be labelled by
a permutation giving the new order of the co-ordinates, relative to
the former order of $(1,2,3)$. For example, the $(1,2,3)$-conjugate
is the square itself and the $(2,1,3)$-conjugate is its matrix
transpose. The $(1,3,2)$-conjugate is found by interchanging columns
and symbols, which is another way of saying that each row, when
thought of as a permutation, is replaced by its inverse. 
A consequence first noted in \cite{MR1670298} is this:

\begin{lem}\label{l:nueven}
  A Latin square is row-Hamiltonian if and only if its
  $(1,3,2)$-conjugate is row-Hamiltonian. Hence,
  $\nu(L)\in\{0,2,4,6\}$ for any Latin square $L$.
\end{lem}

\emph{Column permutations} and \emph{symbol permutations}
can be defined similarly to row permutations, and the operations of conjugacy
interchange these objects. In particular, the row permutations of $L$
correspond to column permutations of the $(2, 1, 3)$-conjugate of $L$,
and to symbol permutations of the $(3, 2, 1)$-conjugate of $L$. This
also leads to notions of column-Hamiltonian and symbol-Hamiltonian
squares where all the column and symbol permutations, respectively, are
single cycles. Another view of \tref{t:RCnotS} is that we are showing
that any two of the three related notions (row-Hamiltonian,
column-Hamiltonian and symbol-Hamiltonian) can be achieved without
achieving the third property.

The case when all three properties hold has been studied in some detail.
If $\nu(L)=6$ for a Latin square $L$, then $L$ is said to be {\em
  atomic} because it has an indivisible structure mimicking the
structure of cyclic groups of prime order.  The
papers~\cite{MR2216455,MR2024246,MR1428074,MR2134185,MR1670298}
construct atomic Latin squares, including several infinite families.
An infinite family of Latin squares for which $\nu=2$ is constructed
in~\cite{MR1899629}. It is easy to construct Latin squares with
$\nu=0$. In particular, any Latin square $L$ of order $n \geq 4$
satisfying $L_{0, 0} = L_{1,1}= 0$ and $L_{0, 1} = L_{1, 0} = 1$ will
have $\nu=0$. By a result of Ryser~\cite{MR42361}, such Latin squares
exist of all orders $n \geq 4$. An explicit example of an infinite
family of Latin squares with $\nu=0$ is the family of addition tables
of the cyclic groups of composite order.  In contrast, almost nothing
was previously known about Latin squares with
$\nu=4$. From~\cite{MR1670298} we know that they have to have odd
order and the smallest example has order $11$. However, the example of
order $11$ in~\cite{MR1670298} was the only published example. This is
one motivation for us to construct an infinite family.

Our construction of the squares in
\tref{t:RCnotS} imitates the construction of atomic Latin squares
in~\cite{MR2134185}. This construction uses so-called cyclotomic
orthomorphisms and diagonally cyclic Latin squares, which we now
define. Let $q$ be a prime power and let $nk = q - 1$, for some
positive integers $n$ and $k$. Let $\gamma$ be a primitive element of
the finite field $\FF$. For $0 \leq j \leq n-1$ we define
$C_j=\{\gamma^{ni+j} : 0 \leq i \leq k - 1\}$ to be a \emph{cyclotomic
  coset} of the unique subgroup $C_0$ of index $n$ in $\FF^*$. A
\emph{cyclotomic map} $\varphi = \varphi_{\gamma}[a_0,\dots,a_{n-1}]$
of index $n$ can then be defined by
\begin{equation}\label{e:cyceqn}
  \varphi(x) = \begin{cases}
    0 &\text{if } x = 0, \\
    a_ix &\text{if } x \in C_i,
  \end{cases}
\end{equation}
where $a_0, \dots, a_{n - 1} \in \FF$.

If $\varphi$ is a cyclotomic map of index $2$ then we call $\varphi$ a
\emph{quadratic map}. In general, the cyclotomic cosets $C_j$ will
depend on the choice of the primitive element $\gamma$. However, we
will almost exclusively be dealing with quadratic maps over $\Z_p$,
the field with $p$ elements. In this case we have $C_0=\Rp_p$, the set
of quadratic residues in the multiplicative group $\Z_p^*$, and
$C_1=\Np_p$, the set of quadratic non-residues in $\Z_p^*$.  In the
quadratic case the choice of $\gamma$ is immaterial, so we omit the
subscript that specifies it. Also, we will find it convenient to index
the rows and columns of our Latin squares using $\Z_p$ so that they
correspond to quasigroups with the same underlying set as $\Z_p$, but
with a different binary operation.

A permutation $\theta : \Z_p \to \Z_p$ is called an
\emph{orthomorphism} if the map $x \mapsto \theta(x) - x$ is also a
permutation of $\Z_p$. If an orthomorphism
$\varphi$ is a cyclotomic map over $\Z_p$, then we call $\varphi$ a
\emph{cyclotomic orthomorphism}. It is well known (see
e.g.~\cite{MR3837138}) that a necessary and sufficient condition for
the quadratic map $\varphi[{a, b}]$ to be an orthomorphism is that
$ab\in\Rp_p$ and $(a-1)(b-1) \in \Rp_p$.

An $n \times n$ matrix $L$ is called diagonally cyclic if $L_{i, j}=k$
implies $L_{i+1, j+1} = k+1$ for any $i, j, k \in \Z_n$. If $L$ is
also a Latin square then we call $L$ a diagonally cyclic Latin
square. It is known that no diagonally cyclic Latin square of even
order exists. This result, in another form, has been known since
Euler. We note the following simple result.

\begin{lem}\label{l:dclsrp}
  Every row permutation of a diagonally cyclic Latin square has an odd
  number of cycles.
\end{lem}

\begin{proof}
  Let $L$ be a diagonally cyclic Latin square and let $r$ be a row
  permutation of $L$. From~\cite[Theorem $12$]{MR2036476} we know that
  $r$ is an even permutation, hence the number of cycles of $r$ of
  even length is even. As the order of $L$ is odd, we must have an odd
  number of cycles of odd length and thus the total number of cycles
  of $r$ is odd.
\end{proof}

A necessary condition for a Latin square to be row-Hamiltonian is that
all row permutations contain an odd number of cycles. So when
searching for row-Hamiltonian Latin squares, \lref{l:dclsrp} tells us
that it makes sense to consider those squares which are diagonally
cyclic. Another major advantage of diagonally cyclic Latin squares is
their large automorphism group (cf.~\lref{l:sufconrh}).

Clearly a diagonally cyclic Latin square is uniquely determined by its
first row. It is known~\cite{MR2036476} that the $p \times p$
diagonally cyclic matrix $L$ with first row defined by $L_{0, j} =
\varphi(j)$ for some $\varphi : \Z_p \to \Z_p$ is a Latin square if
and only if $\varphi$ is an orthomorphism of $\Z_p$. In this case we
say that $L$ is the Latin square generated by $\varphi$. If
$\varphi_{\gamma}[a_0, a_1, \ldots, a_{n-1}]$ is a cyclotomic
orthomorphism of $\Z_p$ then we denote the Latin square generated by
$\varphi_{\gamma}[a_0, a_1, \ldots, a_{n-1}]$ by
$\mathcal{L}_\gamma[a_0, a_1, \ldots, a_{n-1}]$.

\section{Row-Hamiltonian Latin squares}\label{s:main}

In this section we will prove \tref{t:RCnotS} by constructing an
infinite family of Latin squares with $\nu=4$. Let $p$ be a prime with
$p \equiv 1 \bmod 8$ or $p \equiv 3 \bmod 8$. Define the Latin square
$\mathscr{L}_p$ to be the $p \times p$ Latin square $\mathcal{L}[-1,2]$.
Explicitly,
\[
(\mathscr{L}_p)_{i, j} = \begin{cases}
  i & \text{if } j = i, \\
  i - (j-i) & \text{if } j-i \in \Rp_p, \\
  i + 2(j-i) & \text{if } j-i \in \Np_p.
\end{cases}
\]
We will show that $\nu(\mathscr{L}_p) = 4$ for all $p$, except for the
special cases where $p \in \{3, 19\}$. It is easy to check that
the squares $\mathscr{L}_{3}$ and $\mathscr{L}_{19}$ are atomic.

\subsection{Proving that $\mathscr{L}_p$ is row-Hamiltonian}\label{ss:conj}

For a positive integer $n$ we denote the group of permutations on the set $\{0, 1, 2, \ldots, n-1\}$ by $\mathcal{S}_n$. By the \emph{cycle structure} of such a permutation we mean a sorted list of the lengths of its cycles. 

The first step in showing that $\mathscr{L}_p$ has $\nu=4$, except when $p \in \{3, 19\}$, is to show that each square $\mathscr{L}_p$ is row-Hamiltonian. 
To do this from the definition we must check that every row permutation of $\mathscr{L}_p$ is a $p$-cycle. However, we can use the large automorphism group of the square to dramatically reduce the workload, via the following result from~\cite{MR2134185}.

\begin{lem}\label{l:sufconrh}
  Let $p$ be an odd prime and let $a, b \in \Z_p$ be such that $ab, (a - 1)(b - 1) \in \Rp_p$.
  \begin{enumerate}[(i)]
  \item If $p \equiv 3 \bmod 4$ then the cycle structure of any row permutation of $\mathcal{L}[a, b]$ is equal to the cycle structure of the permutation $r_{0, 1}$ of $\mathcal{L}[a, b]$,
  \item If $p \equiv 1 \bmod 4$ then let $c \in \Np_p$. The cycle structure of any row permutation of $\mathcal{L}[a, b]$ is equal to the cycle structure of either the permutation $r_{0,1}$ or $r_{0, c}$ of $\mathcal{L}[a, b]$.
  \end{enumerate}
  Similar statements hold for column and symbol permutations of $\mathcal{L}[a, b]$.
\end{lem}

We note that \lref{l:sufconrh} cannot be improved in the case where $p \equiv 1 \bmod 4$. For example, the row permutation $r_{0,1}$ of the $17 \times 17$ Latin square $\mathcal{L}[14, 10]$ is a $17$-cycle, but the square is not row-Hamiltonian. 

To use \lref{l:sufconrh} to conclude that $\mathscr{L}_p$ is row-Hamiltonian when $p \equiv 1 \bmod 8$, we must show that the permutation $r_{0, c}$ of $\mathscr{L}_p$ is a $p$-cycle, for some $c \in \Np_p$.
It would be inconvenient for our calculations to not have an explicit value for $c$, so we will instead use a different set of sufficient conditions for a Latin square generated by a quadratic orthomorphism to be row-Hamiltonian. From~\cite{MR2134185} we know that the Latin squares $\mathcal{L}[a, b]$ and $\mathcal{L}[b, a]$ are isomorphic via the isomorphism $x \mapsto cx$, for any $c \in \Np_p$. Combining this with \lref{l:sufconrh} we get the following result.

\begin{lem}\label{l:altsufconrh}
  Let $p$ be an odd prime and let $a, b \in \Z_p$ be such that $ab, (a-1)(b-1) \in \Rp_p$. 
  \begin{enumerate}[(i)]
  \item If $p \equiv 3 \bmod 4$ then the Latin square $\mathcal{L}[a, b]$ is row-Hamiltonian if the permutation $r_{0,1}$ of $\mathcal{L}[a, b]$ is a $p$-cycle.
  \item If $p \equiv 1 \bmod 4$ then the Latin square $\mathcal{L}[a, b]$ is row-Hamiltonian if the permutation $r_{0,1}$ of $\mathcal{L}[a, b]$ and the permutation $r_{0,1}$ of $\mathcal{L}[b, a]$ are both $p$-cycles.
  \end{enumerate}
\end{lem}

\subsubsection{A linear algebra approach}\label{ss:linalg}

One goal for the remainder of this subsection is to prove that the permutation $r_{0, 1}$ of $\mathscr{L}_p$ is a $p$-cycle whenever $p \equiv 1 \bmod 8$ or $p \equiv 3 \bmod 8$. We will also prove that the permutation $r_{0, 1}$ of the $p \times p$ Latin square $\mathcal{L}[2, -1]$ is a $p$-cycle for $p \equiv 1 \bmod 8$. To assist in these tasks, we will use a linear algebraic method to determine the cycle structure of a permutation. This method was developed by Beck~\cite{MR442065}, extending the work of Cohn and Lempel~\cite{MR313073}. 

For a positive integer $n$, let $\tau_n$ denote the $n$-cycle,
 $ 
  \tau_n = (0, 1, 2, \ldots, n - 1) \in \mathcal{S}_n.
 $ 
We will need the following lemma concerning products of permutations with the $n$-cycle $\tau_n$. Note that throughout this paper we compose permutations left to right.

\begin{lem}\label{l:matdef}
  Let $a, b \in \{1, 2, \ldots, n - 1\}$ be distinct. Then the permutation
 $ 
  (0, a, b) \cdot \tau_n
 $ 
  is an $n$-cycle if and only if $a < b$.
\end{lem}

\begin{proof}
  Suppose that $a < b$. Then we have that
  \begin{equation*}
    (0, a, b) \cdot \tau_n = (0, a + 1, a + 2, \ldots, b, 1, 2, \ldots, a, b + 1, b + 2, \ldots, n - 1),
  \end{equation*}
  which is an $n$-cycle. Alternatively, if $a > b$ then
 $ 
  (0, a + 1, a + 2, \ldots, n - 1)
 $ 
  is a cycle in $(0, a, b) \cdot \tau_n$ of length less than $n$, so $(0, a, b) \cdot \tau_n$ is not an $n$-cycle.
\end{proof}

Let $T_n \subseteq \mathcal{S}_n$ denote the set of transpositions in $\mathcal{S}_n$. Then we define a map $\wedge : T_n \times T_n \to \Z_2$ via
\begin{equation*}
  \alpha \wedge \beta = \begin{cases}
    1 & \text{if } \alpha \cdot \beta \cdot \tau_n \text{ is an } n \text{-cycle}, \\
    0 & \text{otherwise}.
  \end{cases}
\end{equation*}
Using this operation, we can assign a matrix to any permutation, which contains information about the cycle structure of that permutation.
\begin{defin}\label{d:linkmat}
  Let $\sigma_1, \ldots, \sigma_m \in T_n$. We define the $m\times m$ \emph{link relation matrix} $L(\sigma_1, \sigma_2, \ldots, \sigma_m)$ of the transpositions $\sigma_1, \sigma_2, \ldots, \sigma_m$ by
  \begin{equation*}
    L(\sigma_1, \sigma_2, \ldots, \sigma_m)_{j, i} = L(\sigma_1, \sigma_2, \ldots, \sigma_m)_{i, j} = \sigma_i \wedge \sigma_j,
  \end{equation*}
  for $1\le i < j\le m$ and $L(\sigma_1, \sigma_2, \ldots, \sigma_m)_{i, i} = 0$ for $i \in \{1, 2, \ldots, m\}$.
\end{defin}

We call a matrix \emph{hollow} if every entry of its main diagonal is $0$. It is clear from this definition that $L(\sigma_1, \sigma_2, \ldots, \sigma_m)$ is a symmetric, hollow matrix over $\Z_2$. The following is a special case of a theorem due to Beck~\cite{MR442065}.
\begin{thm}\label{t:linkmatcyc}
  The permutation $\sigma_1 \cdot \sigma_2 \cdots \sigma_m \cdot \tau_n$ is an $n$-cycle if and only if $L(\sigma_1, \sigma_2, \ldots, \sigma_m)$ is non-singular over $\Z_2$.
\end{thm}

\subsubsection{Defining link relation matrices}\label{ss:3mod8}

In this subsection, we will construct the link relation matrix for the permutation $r_{0, 1}$ of $\mathscr{L}_p$ for primes $p \equiv 1 \bmod 8$ or $p \equiv 3 \bmod 8$. We also construct the link relation matrix for the permutation $r_{0, 1}$ of the $p \times p$ Latin square $\mathcal{L}[2, -1]$ when $p \equiv 1 \bmod 8$. First, consider the permutation $r_{0, 1}$ of $\mathscr{L}_p$. Let $f_p$ denote the quadratic orthomorphism $\varphi[{-1, 2}]$ over $\syms$. Then we have that
\begin{equation*}
  (\mathscr{L}_p)_{i, j} = i + f_p(j - i).
\end{equation*}
So the permutation $r_{0, 1}$ of $\mathscr{L}_p$ satisfies $r_{0, 1}(j) = f_p(f_p^{-1}(j) - 1) + 1$. Let 
$\phi_p$ denote the permutation $r_{0, 1}$ of $\mathscr{L}_p$. Then $\phi_p : \Z_p \to \Z_p$ is defined by
 $ 
  \phi_p = f_p^{-1} \cdot \tau_p^{-1} \cdot f_p \cdot \tau_p.
 $ 
We wish to show the following conjecture, which appeared in~\cite{MR2563050}.

\begin{conj}
  For any prime $p \equiv 1 \bmod 8$ or $p \equiv 3 \bmod 8$, the permutation $\phi_p$ is a $p$-cycle.
\end{conj}

Before constructing the link relation matrix for the permutation $\phi_p$, we will need to define some notation.
For an integer $k$, we define $q(k)$ to be the least non-negative integer $x$ such that $x\equiv k\bmod p$.
We will need to use inequalities in $\Z_p$, which will be denoted by $\prec, \preccurlyeq, \succ$ and $\succcurlyeq$. To apply the inequality $a \prec b$ for elements $a, b \in \Z_p$, we simply apply the inequality $q(a) < q(b)$, where $<$ is an inequality in the integers. We apply the inequalities $\preccurlyeq$, $\succ$ and $\succcurlyeq$ in the same way.

We can now work towards constructing the link relation matrix for $\phi_p$. Define the map $\psi_p : \Z_p \to \Z_p$ by
\begin{equation*}
  \psi_p = f_p^{-1} \cdot \tau_p^{-1} \cdot f_p = (0, f_p(-1), f_p(-2), \ldots, f_p(1)),
\end{equation*}
so that $\phi_p = \psi_p \cdot \tau_p$. For $i \in \Z_p^*$ let $x_i = f_p(-i)$. When $p \equiv 3 \bmod 8$ we have,
\[
x_i=\begin{cases}
	-2i & \text{if } i \in \Rp_p, \\
	i & \text{if } i \in \Np_p.
\end{cases}
\]
When $p \equiv 1 \bmod 8$ we have,
\[
x_i = \begin{cases}
	i & \text{if } i \in \Rp_p, \\
	-2i & \text{if } i \in \Np_p.
\end{cases}
\]
Denote the transposition $(0, x_i) \in \mathcal{S}_p$ by $\sigma_i$.
Then we can write
\begin{equation*}
  \phi_p = \psi_p\cdot\tau_p = \left(\displaystyle\prod_{i = 1}^{p - 1} \sigma_i\right) \cdot \tau_p.
\end{equation*}
So we have written $\phi_p$ in the form required to use \tref{t:linkmatcyc}. We can now write the link relation matrix $L(\sigma_1, \sigma_2, \ldots, \sigma_{p - 1})$ in another form.

\begin{lem}\label{l:linkmatdef}
  The link relation matrix $L(\sigma_1, \sigma_2, \ldots, \sigma_{p - 1})$ is the $(p-1) \times (p-1)$ matrix given by
  \begin{equation}\label{e:link}
    L(\sigma_1, \sigma_2, \ldots, \sigma_{p - 1})_{j, i} = L(\sigma_1, \sigma_2, \ldots, \sigma_{p - 1})_{i, j} = \begin{cases}
      1 & \text{if } x_i \prec x_j, \\
      0 & \text{if } x_i \succcurlyeq x_j,
    \end{cases}
  \end{equation}
  for $i, j \in \Z_p^*$ with $i \preccurlyeq j$.
\end{lem}

\begin{proof}
  Let $i, j \in \Z_p^*$ with $i \prec j$. By \dref{d:linkmat}, we have that
  \[
  L(\sigma_1, \sigma_2, \ldots, \sigma_{p - 1})_{j, i} = L(\sigma_1, \sigma_2, \ldots, \sigma_{p - 1})_{i, j} = \begin{cases}
    1 & \text{if } \sigma_i \cdot \sigma_j \cdot \tau_p \text{ is a } p \text{ cycle}, \\
    0 & \text{otherwise}.
  \end{cases}
  \]
  Also
 $ 
  \sigma_i \cdot \sigma_j \cdot \tau_p = (0, x_i, x_j) \cdot \tau_p,
 $ 
  which is a $p$-cycle if and only if $x_i \prec x_j$, by \lref{l:matdef}. So it follows that \eref{e:link} holds for all $i \prec j$. Also, from \dref{d:linkmat} we have that $L(\sigma_1, \sigma_2, \ldots, \sigma_{p - 1})_{i, i} = 0$, which agrees with \eref{e:link} because $x_i \succcurlyeq x_i$ for all $i \in \Z_p^*$.
\end{proof}

When $p \equiv 1 \bmod 8$ we are also interested in the link relation matrix for the permutation $r_{0, 1}$ of the $p \times p$ Latin square $\mathcal{L}[2, -1]$. 
Define
\[
y_i = \begin{cases}
  -2i & \text{if } i \in \Rp_p, \\
  i & \text{if } i \in \Np_p,
\end{cases}
\]
and let $\rho_i = (0, y_i)$. Then by using similar arguments as used to construct the link relation matrix $L(\sigma_1, \sigma_2, \ldots, \sigma_{p-1})$ we can prove that the link relation matrix $L(\rho_1, \rho_2, \ldots, \rho_{p-1})$ for the permutation $r_{0, 1}$ of $\mathcal{L}[2, -1]$ is defined by
\[
L(\rho_1, \rho_2, \ldots, \rho_{p - 1})_{j, i} = L(\rho_1, \rho_2, \ldots, \rho_{p - 1})_{i, j} = \begin{cases}
  1 & \text{if } y_i \prec y_j, \\
  0 & \text{if } y_i \succcurlyeq y_j,
\end{cases}
\]
for $i, j \in \Z_p^*$ with $i \preccurlyeq j$.

\subsubsection{Proving that the link relation matrices are non-singular}\label{ss:ap}

In this subsection we will prove that the link relation matrices constructed in \sref{ss:3mod8} are non-singular over $\Z_2$. Let $p \equiv 1 \bmod 8$ or $p \equiv 3 \bmod 8$ be a prime. We will assume that $p>3$ because some parts of the proof require this. Let $\M \subseteq \Z_p^*$ and let $\N = \Z_p^* \setminus \M$. For $i \in \Z_p^*$ define
\[
a_i = \begin{cases}
  -2i & \text{if } i \in \M\!, \\
  i & \text{if } i \in \N.
\end{cases}
\]
Define a matrix $\C$ by
\[
(\C)_{j, i} = (\C)_{i, j} = \begin{cases}
  1 & \text{if } a_i \prec a_j, \\
  0 & \text{if } a_i \succcurlyeq a_j,
\end{cases}
\]
for $i, j \in \Z_p^*$ with $i \preccurlyeq j$. From now on we will be treating the following three cases:\\[0.5ex]
Case 1: $p \equiv 3 \bmod 8$, $\M = \Rp_p$ and $\N = \Np_p$ where $\C = L(\sigma_1, \sigma_2, \ldots, \sigma_{p-1})$.\\[0.5ex]
Case 2: $p \equiv 1 \bmod 8$, $\M = \Rp_p$ and $\N = \Np_p$ where $\C = L(\rho_1, \rho_2, \ldots, \rho_{p-1})$.\\[0.5ex]
Case 3: $p \equiv 1 \bmod 8$, $\M = \Np_p$ and $\N = \Rp_p$ where $\C = L(\sigma_1, \sigma_2, \ldots, \sigma_{p-1})$.

\medskip

Combining \lref{l:altsufconrh} and \tref{t:linkmatcyc}, we have:

\begin{lem}\label{l:matrh}
  Let $p>3$ be prime with $p \equiv 1 \bmod 8$ or $p \equiv 3 \bmod 8$. To show that $\mathscr{L}_p$ is row-Hamiltonian when $p \equiv 3 \bmod 8$, it suffices to show that the matrix $\C$ is non-singular over $\Z_2$ in Case $1$. To show that $\mathscr{L}_p$ is row-Hamiltonian when $p \equiv 1 \bmod 8$, it suffices to show that the matrix $\C$ is non-singular over $\Z_2$ in Case $2$ and Case $3$. 
\end{lem}

We will show that $\C$ is non-singular in all three cases, using a single argument for which only minor details change between the different cases. We start by defining
 $ 
\kappa_p = \min\{q(i) : i \in \M\}.
 $ 
It is clear in Case $1$ and Case $2$ that $\kappa_p = 1$. However in Case $3$ we know that $\kappa_p$ is the least quadratic non-residue modulo $p$. Note that $\kappa_p$ is also odd in this case because it must be prime, and $2\in \Rp_p$.
We will also need to use the fact that $\kappa_p <p / 5$. This is obviously true in Case $1$ and Case $2$. In Case $3$ when $p = 17$ we know that $\kappa_p = 3 < 17/5$. Also, $\kappa_p < p/5$ for all other primes $p \equiv 1 \bmod 8$ as a result of~\cite[Theorem $2$]{MR3095735}. 

From now on it will be convenient to identify $\Z_p^*$ with the set of
integers $\{1,2,\dots,p-1\}$, so that $q(i)=i$ for all $i\in\Z_p^*$.
It then follows that for $a, b \in \Z_p^*$ the inequalities $a \prec b$, $a \preccurlyeq b$, $a \succ b$ and $a \succcurlyeq b$ are simply $a < b$, $a \leq b$, $a > b$ and $a \geq b$ respectively.

We will often be using the quantity $h_p=(p-1)/2$ throughout this subsection. The following is a simple result which we will use frequently.

\begin{lem}\label{l:abineq}
  Let $a, b \in \Z_p^*$ with $a < b$. Then $-2b \prec -2a$ if and only if $a < b \leq h_p$, $b > a > h_p$ or $b - a > h_p$.
\end{lem}

\begin{proof}
  If $a < b \leq h_p$ then $1 < 2a < 2b \leq p - 1$ and so $q(-2a) = p - 2a>p-2b=q(-2b)$. Alternatively, if $b > a > h_p$ then $q(-2a) = 2p - 2a>2p-2b=q(-2b)$. In either case, $-2b \prec -2a$. It remains to consider the case when $a\le h_p<b$, so that $q(-2a)= p - 2a$ and $q(-2b) = 2p - 2b$. In this case,
  \[
  b - a > h_p \;\Longleftrightarrow\; 2b - 2a \geq p + 1
  \;\Longleftrightarrow\; p - 2a \geq (2p - 2b) + 1 \;\Longleftrightarrow\; -2a \succ -2b.
  \]
  The result follows.
\end{proof}

For a non-empty set $X \subseteq \Z_p$ we define $\min(X)$ to be the unique element $x \in X$ such that $x \preccurlyeq y$ for all $y \in X$. We define $\max(X)$ similarly. We will now work towards showing that $\C$ is non-singular over $\Z_2$ by showing some properties of $\C$. In some situations, the results will depend on whether we are considering Case $1$, Case $2$ or Case $3$. In these situations we will specify which case(s) we are considering. If we do not specify a particular case, then it is understood that the result applies to all three cases.

\begin{defin}
  Let $A$ be a matrix and let $\alpha, \beta$ be subsets of its set of
  row and column indices, respectively. We define $A[\alpha, \beta]$
  to be the $|\alpha| \times |\beta|$ submatrix of $A$ induced by the
  rows in $\alpha$ and columns in $\beta$. If $\alpha = \beta$ then we
  denote the \emph{principal submatrix} $A[\alpha, \alpha]$ by
  $A_\alpha$.
\end{defin}

We note the following easy observation about principal submatrices of $\C$.

\begin{lem}\label{l:princsh}
  Let $I \subseteq \Z_p^*$. Then $\C_I$ is hollow and symmetric.
\end{lem}

We will next show two more interesting lemmas about principal submatrices of $\C$. First, we need the following definitions.

\begin{defin}
  Let $I \subseteq \Z_p^*$ and let $i, j \in I$ with $i \prec j$. We call $i$ and $j$ \emph{adjacent} in $I$ if there is no $k \in I$ such that $i \prec k \prec j$.
\end{defin}

\begin{defin}\label{d:diff}
  Let $I \subseteq \Z_p^*$ and let $i, j \in I$ with $a_i \prec a_j$. We define
  \[
  X(I, i, j) = X(I, j, i) = \{k \in I : a_i \prec a_k \prec a_j\}.
  \]
\end{defin}

\begin{lem}\label{l:red2}
  Let $I \subseteq \Z_p^*$ and suppose $\{i, i+1\}\subseteq I \cap \N$. Then $(\C_I)_{i, k} = (\C_I)_{i+1, k}$ if and only if $k \not\in \{i, i+1\}$. Furthermore $(\C_I)_{i, i} = (\C_I)_{i+1, i+1} = 0$ and $(\C_I)_{i, i+1} = (\C_I)_{i+1, i} = 1$.
\end{lem}

\begin{proof}
  We know that $(\C_I)_{i, i} = (\C_I)_{i+1, i+1} = 0$ from \lref{l:princsh}. As $\{i, i+1\} \subseteq \N$ we know that $a_i = i$ and $a_{i+1} = i+1$. So it follows from the definition of $\C_I$ that $(\C_I)_{i+1, i} = (\C_I)_{i, i+1} = 1$. Now let $k \in I \setminus \{i, i+1\}$.
Since $a_k\notin\{a_i,a_{i+1}\}$, we must have that either $a_k \prec a_i \prec a_{i+1}$ or $a_k\succ a_{i+1} \succ a_{i}$ and it follows from the definition of $\C_I$ that $(\C_I)_{i, k} = (\C_I)_{i+1, k}$.
\end{proof}

\begin{lem}\label{l:red1}
  Let $I \subseteq \Z_p^*$ and let $i, j \in I \cap \M$ be adjacent in $I$ with either $i < j \leq h_p$ or $j > i > h_p$. Suppose that $X(I, i, j) = \{k\}$ for some $k \in I$. Then $(\C_I)_{i, l} = (\C_I)_{j, l}$ if and only if $l \neq k$. If $k < i < j$ then $(\C_I)_{i, k} = 1$ and $(\C_I)_{j, k} = 0$. Otherwise $(\C_I)_{i, k} = 0$ and $(\C_I)_{j, k} = 1$.
\end{lem}

\begin{proof}
  We first note that as $i < j \leq h_p$ or $j > i > h_p$ it follows from \lref{l:abineq} that $a_j \prec a_i$ and hence $(\C_I)_{j, i} = (\C_I)_{i, j} = 0$. Also, $(\C_I)_{i, i} = (\C_I)_{j, j} = 0$ from \lref{l:princsh}. Now let $l \in I \setminus \{i, j, k\}$. Then, as $i, j$ are adjacent in $I$, it follows that either $l < i < j$ or $l > j > i$. Also as $l \not\in X(I, i, j)$ we know that $a_l \prec a_j \prec a_i$ or $a_l \succ a_i \succ a_j$. Hence $(\C_I)_{i, l} = (\C_I)_{j, l}$. 
  
  Now $i, j$ are adjacent in $I$ and $k \in X(I,i,j)$, so we know that $a_j \prec a_k \prec a_i$ and either $k < i < j$ or $k > j > i$. If $k < i < j$ then $(\C_I)_{i, k} = 1$ and $(\C_I)_{j, k} = 0$ by the definition of $\C_I$. Similarly, if $k > j > i$ then $(\C_I)_{i, k} = 0$ and $(\C_I)_{j, k} = 1$.
\end{proof}

\lref{l:red2} and \lref{l:red1} motivate the following definition.

\begin{defin}\label{d:reduction}
  Let $I \subseteq \Z_p^*$.
  \begin{itemize}
  \item Let $i, i+1 \in I \cap \N$. Then we can define $J$ to be the set $I \setminus \{i, i+1\}$. We say that $J$ has been obtained by performing a Type One reduction of $I$ on $i$ and $i+1$. We also say that we have deleted $i$ and $i+1$ in this process.
  \item Let $i, j \in I \cap \M$ be adjacent in $I$ with either $i < j \leq h_p$ or $j > i > h_p$, so that $a_j \prec a_i$. Further suppose that $X(I, i, j) = \{k\}$ for some $k \in I$. If $k < i < j$ then we define $J$ to be the set $I \setminus \{i, k\}$. We say that we have deleted $i$ and $k$ in this process. If $k > j > i$ then we define $J$ to be the set $I \setminus \{j, k\}$. We say that we have deleted $j$ and $k$ in this process. In both situations, we say that $J$ has been obtained by performing a Type Two reduction of $I$ on $i$ and $j$.
  \end{itemize}
\end{defin}

For example, consider the case when $p = 11$ and let $I = \Z_p^*$. We have that $3, 4 \in I \cap \Rp_p$ are adjacent and $X(I, 3, 4) = \{9\}$. So by defining $J = \Z_p^* \setminus \{4, 9\}$ we are performing a Type Two reduction of $I$ on $3$ and $4$. The following is the initial matrix $\C_{I}$, followed by the matrix $\C_{J}$ after performing the reduction operation. Note that the highlighted third and fourth rows of $\C_{I}$ agree on every column except column $9$. This is predicted by \lref{l:red1}.
\begin{equation*}
  \begin{pNiceMatrix}
    0&0&0&0&0&0&0&0&0&1\\
    0&0&1&1&0&1&1&1&1&1\\
    \rowcolor{blue!15}
    0&1&0&0&0&1&1&1&0&1\\
    \rowcolor{blue!15}
    0&1&0&0&0&1&1&1&1&1\\
    0&0&0&0&0&1&1&1&1&1\\
    0&1&1&1&1&0&1&1&0&1\\
    0&1&1&1&1&1&0&1&0&1\\
    0&1&1&1&1&1&1&0&0&1\\
    0&1&0&1&1&0&0&0&0&1\\
    1&1&1&1&1&1&1&1&1&0
  \end{pNiceMatrix}
  \quad\longrightarrow\quad
  \begin{pNiceMatrix}
    0&0&0&0&0&0&0&1\\
    0&0&1&0&1&1&1&1\\
    0&1&0&0&1&1&1&1\\
    0&0&0&0&1&1&1&1\\
    0&1&1&1&0&1&1&1\\
    0&1&1&1&1&0&1&1\\
    0&1&1&1&1&1&0&1\\
    1&1&1&1&1&1&1&0
  \end{pNiceMatrix}
\end{equation*}

The interest in these reduction operations is due to the following lemma.

\begin{lem}\label{l:detpres}
  Let $I \subseteq \Z_p^*$ and suppose that $J$ is obtained by performing a Type One or Type Two reduction of $I$ on $i$ and $j$ for some $i, j \in I$. Then,
  $ 
  \det(\C_I) \equiv \det(\C_J) \bmod 2.
  $ 
\end{lem}

\begin{proof}
  We first prove the claim for a Type One reduction. Suppose that we performed the reduction on $i$ and $i+1$, for some $i, i+1 \in I \cap \N$. By \lref{l:red2}, applying the transposition $(i, i+1)$ to the rows and columns of $\C_I$ reveals a symmetry of $\C_I$. The action of this symmetry partitions the positive diagonals (i.e.~selections of $|I|$ ones from different rows and columns) of $\C_I$ into orbits of size $1$ or $2$. The only orbits of size $1$ are the positive diagonals that include both cells $(i, i+1)$ and $(i+1, i)$. The number of such diagonals is equal to the determinant of $\C_J$ modulo $2$. Hence $\det(\C_I) \equiv \det(\C_J) \bmod 2$.
  
  The argument for Type Two reductions is similar. Suppose that we performed the reduction on $i$ and $j$ for some adjacent $i, j \in I \cap \M$, and assume, without loss of generality, that $i < j$. We can write $X(I, i, j) = \{k\}$ for some $k \in I$. Suppose in the first instance that $k < i < j$, so that $J = I \setminus \{i, k\}$. Since $\C_I$ is symmetric, we see that $\det(\C_I) \equiv |V_I| \bmod 2$, where $$V_I = \{\sigma \in \text{Sym}(I) : \sigma = \sigma^{-1} \text{ and } (\C_I)_{l,\sigma l} = 1 \text{ for all } l \in I\}.$$ Also define the set $Z_I = \{\sigma \in V_I : \sigma(i) = k\}$. \lref{l:red1} shows that applying the transposition $(i, j)$ to the rows and columns of $\C_I$ induces a fixed-point-free involution on $V_I\setminus Z_I$. Hence $\det(\C_I) \equiv |Z_I| \equiv \det(\C_J) \bmod 2$. The case where $k > j > i$ is dealt with similarly.
\end{proof}

We will be interested in sets which are obtained from $\Z_p^*$ by a sequence $R_1,R_2,\dots,R_n$ of reduction operations. 
For $m \in \{0, 1, 2, \ldots, n\}$ we will let $I_m$ denote the set obtained from $\Z_p^*$ by the first $m$ of these reduction operations. So $I_0 = \Z_p^*$. We will sometimes use $I$ as shorthand for $I_n$.
\lref{l:detpres} shows that $\det(\C_I) \bmod 2$ is an invariant.
Our next lemma gives another useful invariant of the set $I$.

\begin{lem}\label{l:eqrowss}
  Let $I \subseteq \Z_p^*$ be obtained from $\Z_p^*$ by a sequence of Type One or Type Two reductions. Let $i, j \in I \cap \M$ with either $i < j \leq h_p$ or $j > i > h_p$. Then $|X(I, i, j)| \equiv 1 \bmod 2$.
\end{lem}

\begin{proof}
  Suppose that $I$ has been obtained from $\Z_p^*$ by a sequence $R_1, R_2, \ldots, R_n$ of reductions. We will prove the statement is true for $I_m$, by induction on $m \in \{0, 1, 2, \ldots, n\}$. When $m=0$ we have $I_m = \Z_p^*$. Let $i, j \in I_m \cap \M$ with either $i < j \leq h_p$ or $j > i > h_p$. If $i < j \leq h_p$ then 
  \[
  q(-2i)-q(-2j)=(p - 2i) - (p - 2j) = 2(j - i) \equiv 0 \bmod 2.
  \]
  So there are an odd number of elements between $a_j$ and $a_i$ in $\Z_p^*$. Hence $|X(I_m, i, j)| \equiv 1 \bmod 2$. The argument when $m=0$ and $j > i > h_p$ is similar.
  
  Now suppose that the result is true for $I_{m-1}$, for some positive integer $m$ and consider the set $I_m$. Let $i, j \in I_m \cap \M$ with $i < j \leq h_p$ or $j > i > h_p$. By the induction hypothesis we know that $|X(I_{m-1}, i, j)| \equiv 1 \bmod 2$. First suppose that $R_m$ is a Type One reduction on some $k, k+1 \in I_{m-1} \cap \N$. As $-2\in \Rp_p$ and $i, j \in \M$, it follows that $-2i, -2j \in \M$. So in particular $k, k+1 \not\in \{-2i, -2j\}$. It follows that $-2j \prec k \prec -2i$ if and only if $-2j \prec k + 1 \prec -2i$. So we either have $X(I_m, i, j) = X(I_{m-1}, i, j)$ or $X(I_m, i, j) = X(I_{m-1}, i, j) \setminus \{k, k+1\}$. Either way it follows that $|X(I_m, i, j)| \equiv |X(I_{m-1}, i, j)| \equiv 1 \bmod 2$.

  Next suppose that $R_m$ is a Type Two reduction on $x$ and $y$ for some $x, y \in I_{m-1} \cap \M$. Assume, without loss of generality, that $x < y$. By \dref{d:reduction} we know that $X(I_{m-1}, x, y) = \{k\}$ for some $k \in I_{m-1}$. We also know that $x < y \leq h_p$ or $x > y > h_p$, so by \lref{l:abineq} we know that $-2y \prec -2x$. Assume first that $k < x < y$ so that $I_m = I_{m-1} \setminus \{k, x\}$. If $k$ and $x$ were both not contained in $X(I_{m-1}, i, j)$, then $X(I_m, i, j) = X(I_{m-1}, i, j)$, which has odd cardinality. So suppose that $k \in X(I_{m-1}, i, j)$. Then, since $k \in X(I_{m-1}, x, y)$, we have $\max\{-2j, -2y\}\prec a_k \prec \min\{-2i, -2x\}$. If $-2y \prec -2j$ then it would follow that $j \in X(I_{m-1}, x, y)$, which is a contradiction. So $-2y \succcurlyeq -2j$. Similarly $-2x \preccurlyeq -2i$. Also, as $x \not\in I_m$ we know that $x \neq i$ and so we have that
 $ 
    -2j \preccurlyeq -2y \prec a_k \prec -2x \prec -2i.
 $ 
  Hence $k, x \in X(I_{m-1}, i, j)$. As $I_m = I_{m-1} \setminus \{x, k\}$ it follows that $X(I_m, i, j) = X(I_{m-1}, i, j) \setminus \{k, x\}$ and so $|X(I_m, i, j)| \equiv |X(I_{m-1}, i, j)| \equiv 1 \bmod 2$. It remains to deal with the case where $x \in X(I_{m-1}, i, j)$ and $k \not\in X(I_{m-1}, i, j)$. In this case we must have
 $ 
  -2y \prec a_k \prec -2j \prec -2x \prec -2i.
 $ 
  But this implies that $j \in X(I_{m-1}, x, y)$, which is a contradiction. So in all cases we have shown that $|X(I_m, i, j)| \equiv 1 \bmod 2$. The situation where $k > y > x$ is dealt with using analogous arguments.
\end{proof}

To prove that the matrix $\C$ is non-singular over $\Z_2$ we will now
consider a specific sequence of reduction operations. There will be
three different stages of these reduction operations. We begin with
$I_0 = \Z_p^*$ and in the first stage we will perform Type One
reductions until we reach a set $I_{t_1}$ which contains no
consecutive elements of $\N$. In the second stage, we will perform
Type Two reductions that delete one element of $\M$ and one element of $\N$,
until we reach a
set $I_{t_2}$ which does not contain any elements of $\N$, except for
$p-1$ in Case $1$. In the last stage we will perform Type
Two reductions that delete two elements of $\M$
until we reach a set $I_{t_3}$ of size $2$. Note that
we must have $t_3 = h_p-1$. For $i \in \{1, 2, 3\}$ we will let $C_i$
denote the principal submatrix of $\C$ induced by the rows and columns
in $I_{t_i}$. It follows from \lref{l:detpres} that
$\det(\C)\equiv \det(\C_3) \bmod 2$. Before detailing the first stage of the
sequence of reductions we need the following definition.

\begin{defin}\label{d:qnrrun}
Let $p$ be a prime. A maximal sequence of consecutive elements of
$\Z_p^*$ which is contained in $\N$ is called a \emph{run of $\N$-elements}.
For $i \in \N$ we define $l_i$ to be the
\emph{length} of the run of $\N$-elements containing $i$ (that is, the
number of elements in the run). We also define $\ell = \max\{l_i : i\in \N\}$.
\end{defin}

\begin{algorithm}[H]
  \DontPrintSemicolon
  \SetKwInOut{Input}{input}\SetKwInOut{Output}{output}
  \Input{A prime $p > 3$ with $p \equiv 1 \bmod 8$ or $p \equiv 3 \bmod 8$}
  $I_0 := \Z_p^*$ \;
  $\zeta := 0$ \;
  \For{each run of $\N$-elements of length at least $2$}{Let $\iota \in \Z_p^*$ be the first element in the run\;
    Let $\lambda \in \Z_p^*$ be such that $\iota+\lambda$ is the last element in the run\;
    \For{odd $\theta \in \{1, 2, \ldots, \lambda\}$}{
      $I_{\zeta+1} := I_{\zeta} \setminus \{\iota+ \lambda - \theta, \iota+ \lambda - \theta + 1\}$ \;
      $\zeta := \zeta+1$
    }
  }
  \Return $\zeta, I_{\zeta}$\;
  \caption{Deleting consecutive elements of $\N$}
\end{algorithm}

Let $t_1$ denote the number of reduction operations performed in Algorithm $1$. We first note that in each iteration of Algorithm $1$ we have that $\iota + \lambda - \theta$, $\iota + \lambda - \theta + 1 \in \N$. Using this and \dref{d:reduction} we get the following result.

\begin{lem}\label{l:mrocor}
  For each $\zeta \in \{0, 1, 2, \ldots, t_1-1\}$ the set $I_{\zeta+1}$ has been obtained by performing a Type One reduction of $I_{\zeta}$.
\end{lem}

We will now prove some facts about the output of Algorithm $1$.

\begin{lem}\label{l:mroprop} \mbox{ }
  \begin{enumerate}[(i)] 
  \item Let $z \in \N$. If $l_z \equiv 0 \bmod 2$ then $z \not\in I_{t_1}$. Otherwise $z \in I_{t_1}$ if and only if $z$ is the first element in the run of $\N$-elements containing $z$,
  \item $\min(I_{t_1}) = \kappa_p$,
  \item $\max(I_{t_1}) = p - \kappa_p$.
  \end{enumerate}
\end{lem}
\begin{proof}
  As $\theta$ takes the odd values in $\{1,2,\dots,2\lceil\lambda/2\rceil-1\}$, we see that $\iota+\lambda-\theta+1$ and $\iota+\lambda-\theta$ run through all elements from $\iota+\lambda$ down to $\iota+1+\lambda-2\lceil\lambda/2\rceil$. Claim $(i)$ follows.
  
  We now prove $(ii)$. By definition $\kappa_p \in \M$ and we know that only elements of $\N$ are deleted in Algorithm $1$. Thus $\kappa_p \in I_{t_1}$. In Case $1$ and Case $2$ we know that $\kappa_p = 1$ so it is clear that $\kappa_p = \min(I_{t_1})$. Now consider Case $3$. We know that $\kappa_p$ is odd so $\{z \in \Z_p^* : z \prec \kappa_p\}$ form the elements of a run of $\N$-elements of even length. Then by $(i)$ we know that all of these elements have been deleted in Algorithm $1$, hence $\kappa_p = \min(I_{t_1})$ as claimed.
  
  Finally, we will prove $(iii)$. In Case $1$ we know that $p-\kappa_p = p-1 \in \Np_p = \N$. We also know that $-2 \in \Rp_p$ so $l_{p-1} = 1$ is odd. So by $(i)$ we know that $p-\kappa_p \in I_{t_1}$, thus $p-\kappa_p = \max(I_{t_1})$. In Case $2$ we have $p-\kappa_p = p-1 \in \Rp_p = \M$ and hence has not been deleted in Algorithm $1$. So it is clear that $p-\kappa_p = \max(I_{t_1})$. In Case $3$ we have that $p-\kappa_p \in \Np_p = \M$ and so $p-\kappa_p \in I_{t_1}$. As $-1 \in \Rp_p$ it follows that the set $\{z \in \Z_p^*: z \succ p-\kappa_p\}$ forms a run of $\N$-elements of even length, hence by $(i)$ every element in this set has been deleted. So again we have that $p-\kappa_p = \max(I_{t_1})$.
\end{proof}

\begin{ex} In Case $1$, when $p=11$ the matrix $\C_1$ is
  \[\begin{pmatrix}
  0&0&0&0&0&0&0&1\\
  0&0&1&1&0&1&1&1\\
  0&1&0&0&0&1&0&1\\
  0&1&0&0&0&1&1&1\\
  0&0&0&0&0&1&1&1\\
  0&1&1&1&1&0&0&1\\
  0&1&0&1&1&0&0&1\\
  1&1&1&1&1&1&1&0\\
  \end{pmatrix}.
  \]
\end{ex}

We can now describe the second stage of the reduction operations.

\begin{algorithm}[H]
  \DontPrintSemicolon
  \SetKwInOut{Input}{input}\SetKwInOut{Output}{output}
  \Input{$t_1$, $I_{t_1}$}
  $\zeta := t_1$ \; 
  \For{odd $\alpha\in \{1, 2, \ldots, \ell\}$ (in increasing order)}{
    \For{$\beta \in (I_{\zeta} \cap \N) \setminus \{p-1\}$ such that $l_\beta = \alpha$}{
      $\gamma_1 := \min\{q(-2^{-1}(\beta-1)), q(-2^{-1}(\beta+\alpha))\}$\;  
      $\gamma_2 := \max\{q(-2^{-1}(\beta-1)), q(-2^{-1}(\beta+\alpha))\}$\;
      Define $\gamma_0$ to be maximal such that $\gamma_0 \leq \gamma_1$ and $\gamma_0 \in I_{\zeta}$\;
      Define $\gamma_3$ to be minimal such that $\gamma_3 \geq \gamma_2$ and $\gamma_3 \in I_\zeta$\;
      \eIf{$\beta <\gamma_0$}{
$I_{\zeta+1} := I_\zeta \setminus \{\beta, \gamma_0\}$\;
      }{
$I_{\zeta+1} := I_\zeta \setminus \{\beta, \gamma_3\}$ \;
      }
      $\zeta := \zeta + 1$
    }
  }
  \Return $\zeta, I_{\zeta}$\;
  \caption{Deleting remaining elements of $\N$ except $p-1$ in Case $1$}
\end{algorithm}

Let $t_2$ denote the number of reduction operations performed in Algorithm $1$ and Algorithm $2$ combined. It is not clear at this stage that $\gamma_0$ and $\gamma_3$ always exist in Algorithm $2$. However this will be shown. We will also show that for each $\zeta \in \{t_1, t_1+1, \ldots, t_2-1\}$ the set $I_{\zeta+1}$ is obtained by performing a Type Two reduction of $I_\zeta$. We will need some preliminary lemmas to prove this.

\begin{lem}\label{l:k1k2}
  In all iterations of the inner {\rm \textbf{for}} loop of Algorithm $2$, $\gamma_1, \gamma_2 \in \M$ and either $\gamma_1 < \gamma_2 \leq h_p$ or $\gamma_2 > \gamma_1 > h_p$. Hence,
 $ 
  -2\gamma_2 \equiv \beta-1 < \beta+\alpha \equiv -2\gamma_1 \bmod p.
 $ 
\end{lem}

\begin{proof}
  Suppose, for a contradiction, that in some iteration of the inner \textbf{for} loop of Algorithm $2$, we have $\gamma_1 \leq h_p<\gamma_2$. Then $q(-2\gamma_1) = p-2\gamma_1$ and $q(-2\gamma_2) = 2p-2\gamma_2$. First suppose that $\gamma_2 - \gamma_1 \leq h_p$. From \lref{l:abineq} it follows that $\beta-1\equiv-2\gamma_1 \prec -2\gamma_2\equiv \beta+\alpha \bmod p$. Hence, we have that
  \begin{equation}\label{e:k1k21}
    -2\gamma_2+2\gamma_1 \equiv \beta+\alpha-(\beta-1) \equiv \alpha+1 \bmod p.
  \end{equation}
  We also have that
  \begin{equation}\label{e:k1k22}
    -2\gamma_2+2\gamma_1 \equiv 2p-2\gamma_2-(p-2\gamma_1) \equiv p-2(\gamma_2-\gamma_1)\bmod p.
  \end{equation}
  Combining \eref{e:k1k21} and \eref{e:k1k22} we see that $p-2(\gamma_2-\gamma_1)=\alpha + 1$ (both sides of this equality are known to lie within $\{1, 2, \ldots, p - 1\}$, by construction). But this is a contradiction of the fact that $\alpha$ is odd.
  
Considering the case where $\gamma_2 - \gamma_1 > h_p$ and following
the same arguments we can show that $\alpha+1 = 2(\gamma_2 - \gamma_1) - p$,
again a contradiction. So it follows that we must always have
either $\gamma_1 < \gamma_2 \leq h_p$ or $\gamma_2 > \gamma_1 > h_p$.
The claim that $-2\gamma_2 \equiv \beta - 1 \bmod p$ and
$-2\gamma_1 \equiv \beta + \alpha \bmod p$ then follows from
\lref{l:abineq}. Finally, we know that $\beta \in I_\zeta \cap \N$ and
so it follows from \lref{l:mroprop} that $\beta$ is the first integer
in a run of $\N$-elements and hence $\beta-1 \in \M$. Similarly, as
$l_\beta = \alpha$ it follows that $\beta+\alpha \in \M$ also.
Since $-2^{-1} \in \Rp_p$, it follows that $\gamma_1, \gamma_2 \in \M$ also.
\end{proof}

In the iteration of the outer \textbf{for} loop of Algorithm $2$ with $\alpha = k$, we delete all elements $z \in I_{t_1} \cap \N$ with $l_z = k$, except for when $z = p-1$ and $k = 1$ in Case $1$.

\begin{lem}\label{l:kdel}
  For all iterations of the inner {\rm \textbf{for}} loop of Algorithm $2$, if $k \in \mathbb{Z}_p^*$ with $\gamma_1 < k < \gamma_2$ then $k \in \N$ and $l_k < \alpha$. In particular, $k \not\in I_\zeta$.
\end{lem}

\begin{proof}
  We have
 $ 
  \beta-1 \equiv -2\gamma_2 \prec -2k \prec -2\gamma_1 \equiv \beta+\alpha \bmod p.
 $ 
  Hence $-2k$ is contained in the run of $\N$-elements beginning at $\beta$. As $-2 \in \Rp_p$ it follows that $k \in \N$.
  Also,
  \[
  \gamma_2 - \gamma_1 \equiv -2^{-1}(\beta-1) + 2^{-1}(\beta+\alpha) \equiv \frac{\alpha+1}2 \bmod p.
  \]
  As $\alpha$ is odd it follows that $(\alpha+1)/2$ is an integer in $\{1, 2, \ldots, p-1\}$. Hence $\gamma_1 < k < \gamma_2$ implies that $l_k \leq (\alpha+1)/2 < \alpha$. But then we must have already deleted $k$.
\end{proof}

\begin{lem}\label{l:tech}
  Let $\zeta \in \{t_1+1, t_1+2, \ldots, t_2-1\}$. Let $j \in \M$ and $i \in I_\zeta$ with $0<i<j<p$. Suppose that, for each $z \in \M$ such that $i < z \leq j$, we have deleted $z$ in one of the first $\zeta$ reductions when performing a Type Two reduction on $z$ and $y$ for some $y \in \M$ with $i \leq y \leq j$. Then $i \in \M$ and $\{z \in I_\zeta  : -2j \preccurlyeq a_z \prec -2i\} = \emptyset$. Furthermore if $j > h_p$ then $i > h_p$.
\end{lem}

\begin{proof}
  The proof will be by induction on $m=\big|\{z \in \M : i < z < j\}\big|$. First consider when $m = 0$. Let $s$ be maximal such that $j \in I_s$, so that $R_{s+1}$ deletes $j$. By assumption we know that $R_{s+1}$ is performed on $k, j$ for some $k \in I_s \cap \M$ such that $i \leq k < j$. Then by \dref{d:reduction} we know that $k, j$ are adjacent in $I_s$. As $m=0$ it follows that $k = i$. So by \dref{d:reduction} we have that $i \in \M$ and if $j > h_p$ then $i > h_p$ also. Furthermore, $X(I_s, i, j) = \{k'\}$ for some $k' \in I_s$, and $I_{s+1} = I_s \setminus \{j, k'\}$. This means that $\{z \in I_{s + 1} : -2j \preccurlyeq a_z \prec -2i\} = \emptyset$. As $I_\zeta  \subseteq I_{s + 1}$, it follows that the claim holds when $m = 0$.
  
  Now suppose the claim is true whenever $m < m'$, and that $\big|\{z \in \M : i < z < j\}\big| = m'$. Let $i, j \in \Z_p^*$ be such that the hypotheses of the lemma are satisfied. Let $u$ be maximal such that there exists some $k \in \{i + 1, \ldots, j\} \cap \M \cap I_u$. Each of the first $t_2$ reductions deletes at most one element of $\M$, and hence at most one element of $\{z \in \M : i < z \leq j\}$. Hence, we must have $\{i + 1, \ldots, j\} \cap \M \cap I_u = \{k\}$. First suppose that $k \neq j$ and note that in this case we must have $u > t_1$ as we have already deleted $j$.  
  Consider $z \in \M$ such that $k < z \leq j$. By assumption we know that $z$ has been deleted in some reduction $R_v$ of Type Two, where $v<u$ and we have performed $R_v$ on $z, y$ for some $y \in \M$ with $i \leq y \leq j$. We know from \dref{d:reduction} that $y$ and $z$ are adjacent in $I_{v-1}$. Also, $k < z$ and by definition $k\in I_u\subset I_{v-1}$. Thus we must have $k \leq y$. 
  So, noting that $\big|\{z \in \M : k < z < j\}\big| < m'$, we can now apply the induction hypothesis with $i = k$ and $\zeta = u$ to see that
  \begin{equation}\label{e:xijsep}
    \{z \in I_u : -2j \preccurlyeq a_z \prec -2k\} = \emptyset
  \end{equation}
  and $k > h_p$ if $j > h_p$.
  From the definition of $u$ and the assumptions of the lemma we know that $R_{u+1}$ is a Type Two reduction on $k, k'$ for some $k' \in I_u \cap \M$ adjacent to $k$ with $i \leq k' \leq j$. But we know that $I_u \cap \M \cap (\{i+1, \ldots, j\} \setminus \{k\}) = \emptyset$ by definition of $k$. This forces $k' = i$, so $R_{u+1}$ is a Type Two reduction on $i, k$. So from \dref{d:reduction} it follows that $i \in \M$ and $i > h_p$ if $k > h_p$, which is true if $j > h_p$. 
  Hence $R_{u+1}$ deletes $k$ and $k''$, where $X(I_u, i, k) = \{k''\}$ for some $k'' \in I_u$. 
  It follows that $\{z \in I_{u+1} : -2k \preccurlyeq a_z \prec -2i\} = \emptyset$. As $I_\zeta \subseteq I_{u+1} \subset I_u$ it follows from \eref{e:xijsep} that $\{z \in I_\zeta  : -2j \preccurlyeq a_z \prec -2i\} = \emptyset$, as required.
  
  It remains to deal with the situation where $k = j$. By assumption and the definition of $k$, we know that $R_{u+1}$ is a Type Two reduction on $j$ and $l$ for some $l \in \M\cap I_u$ with $i \leq l \leq j$. From the definition of $k$, we know that $\{i+1, \ldots, j - 1\} \cap \M \cap I_u=\emptyset$, hence $i = l\in \M$. Furthermore, as we are performing a Type Two reduction on $i$ and $j$ it follows from \dref{d:reduction} that if $j > h_p$ then $i > h_p$ also. We also know that $X(I_u, i, j) = \{k'\}$ for some $k' \in I_u$, hence $I_{u+1} = I_u \setminus \{j, k'\}$. Therefore, $\{z \in I_{u+1} : -2j \preccurlyeq a_z \prec -2i\} = \emptyset$. Finally, we note that $I_\zeta  \subseteq I_{u+1}$ and the claim then follows.
\end{proof}

Using analogous arguments as in the proof of \lref{l:tech}, we can prove the following.

\begin{lem}\label{l:tech2}
  Let $\zeta \in \{t_1+1, t_1+2, \ldots, t_2-1\}$. Let $j \in \M$ and $i \in I_\zeta$ with $0<j<i<p$. Suppose that, for each $z\in \M$ such that $j \leq z < i$, we have deleted $z$ in one of the first $\zeta$ reductions when performing a Type Two reduction on $z$ and $y$ for some $y \in \M$ with $j \leq y \leq i$. Then $i \in \M$ and $\{z \in I_\zeta  : -2i \prec a_z \preccurlyeq -2j\} = \emptyset$. Furthermore, if $j \leq h_p$, then $i \leq h_p$ also.
\end{lem}

We are now ready to prove that Algorithm 2 performs Type Two reductions.

\begin{lem}\label{l:mrtred}
  For each $n \in \{t_1, t_1+1, \ldots, t_2-1\}$, the following properties hold in the iteration of the inner {\rm \textbf{for}} loop of Algorithm $2$ with $\zeta = n$:
  \begin{enumerate}[(i)]
  	\item $\kappa_p, p-\kappa_p \in I_n$,
  	\item $\gamma_0$ and $\gamma_3$ are well defined,
  	\item $\gamma_0, \gamma_3 \in \M$ are adjacent in $I_n$ with $X_{I_n, \gamma_0, \gamma_3} = \{\beta\}$ and either $\gamma_0 < \gamma_3 \leq h_p$ or $\gamma_3 > \gamma_0 > h_p$.
  \end{enumerate}
Therefore $I_{n+1}$ has been obtained from $I_n$ by a Type Two reduction on $\gamma_0$ and $\gamma_3$.
\end{lem}

\begin{proof}
	Suppose, for a contradiction, that the lemma is not true. Let $\zeta = n$ in the first iteration of the inner \textbf{for} loop of Algorithm $2$ when one of the claims does not hold. We first suppose that $(i)$ is not true. By minimality of $n$ we know that $R_n$ deletes $\kappa_p$ or $p-\kappa_p$. First suppose that $R_n$ deletes $\kappa_p$. Then when $\zeta = n-1$ we have $\kappa_p \in \{\gamma_0, \gamma_3, \beta\}$. As $\kappa_p \in \M$ and $\beta \in \N$ it follows that $\kappa_p \neq \beta$. Also, \lref{l:mroprop} implies that $\kappa_p$ is the minimum element of $I_\zeta$, thus we must have $\kappa_p = \gamma_0$. But we only delete $\gamma_0$ if $\beta < \gamma_0 = \kappa_p$ which is impossible. Now suppose that $R_n$ deletes $p-\kappa_p$. Consider the iteration of the inner \textbf{for} loop with $\zeta=n-1$. We will examine Case $1$ separately from Case $2$ and Case $3$. In Case $1$, $p-\kappa_p = p-1 \in \N$. We know that $\beta \neq p-1$ and we also know that $\gamma_0, \gamma_3 \in \M$ by minimality of $n$. Hence $p-\kappa_p \not\in \{\gamma_0, \gamma_3, \beta\}$ and thus $p-\kappa_p$ cannot have been deleted. Now consider Case $2$ and Case $3$. We know from \lref{l:mroprop} and its proof that $p-\kappa_p \in \M$ and that $p-\kappa_p$ is the maximum element of $I_{n-1}$. Since $\beta \in \N$ and $p-\kappa_p = \max(I_{n-1})$ we must have $\gamma_3 = p-\kappa_p$. Since $\gamma_3$ is deleted we have $\beta \geq \gamma_0$. As $\beta \in \N$ and $\gamma_0 \in \M$ we have $\beta > \gamma_0$, thus $\beta > \gamma_1$ by definition of $\gamma_0$. As $\gamma_3$ is the maximum element of $I_{n-1}$ we have $\beta < \gamma_3$, hence $\beta < \gamma_2$ also. But then \lref{l:kdel} implies that $\beta$ must have already been deleted, a contradiction which proves $(i)$.
	
We know from \lref{l:k1k2} that $\gamma_1,\gamma_2\in\M$. By definition, there are no elements of $\M$ below $\kappa_p$. Also, from the proof of \lref{l:mroprop}$(iii)$ it is apparent that no element of $\M$ lies above $p-\kappa_p$. Therefore $\kappa_p\preccurlyeq\gamma_1\preccurlyeq\gamma_2\preccurlyeq p-\kappa_p$. Part $(ii)$ now follows from part $(i)$.
  
  Finally we will show that $(iii)$ must hold. The claim that $\gamma_0$ and $\gamma_3$ are adjacent in $I_n$ follows directly from \lref{l:kdel}. We know from \lref{l:k1k2} that either $\gamma_1 < \gamma_2 \leq h_p$ or $\gamma_2 > \gamma_1 > h_p$. First suppose that $\gamma_2 > \gamma_1 > h_p$. If $\gamma_1 \in I_{n}$ then we know that $\gamma_0 = \gamma_1 \in \M$, hence $\gamma_0 > h_p$. Otherwise, we have deleted all elements in $\{\gamma_0+1, \ldots, \gamma_1\}$. Let $k' \in \{\gamma_0+1, \ldots, \gamma_1\} \cap \M$. Note that in Algorithm $1$ we only deleted elements of $\N$, hence $k'$ must have been deleted in some iteration of the inner \textbf{for} loop in Algorithm $2$. Let $m \in \{t_1, t_1+1, \ldots, n-1\}$ be such that $k'$ was deleted in reduction $R_{m+1}$. So we perform a Type Two reduction of $I_m$ on $k$ and $k'$, for some $k \in I_m \cap \M$ with $k$ and $k'$ adjacent in $I_m$. As $k' > \gamma_0$ with $\gamma_0 \in I_{n} \subseteq I_m$ and $k, k'$ are adjacent in $I_m$ it follows that we must have $k \geq \gamma_0$. Suppose, for a contradiction, that $k > \gamma_1$. Then $k \geq \gamma_2 > h_p+1$ by \lref{l:kdel}. So, as we perform a Type Two reduction on $k$ and $k'$, it follows from \dref{d:reduction} that $k' > h_p$ also. So $q(-2k) = 2p - 2k$ and $q(-2k') = 2p - 2k'$. Then, as $k' \leq \gamma_1 < \gamma_2 \leq k$, it follows that
  \[
  2p - 2k \leq 2p - 2\gamma_2 = \beta - 1 < \beta + \alpha = 2p - 2\gamma_1 \leq 2p - 2k'.
  \]
  So $-2k \prec \beta \prec -2k'$ hence $\beta \in X(I_m, k, k')$. So when we perform the reduction on $k, k'$ we delete $\beta$, but this is a contradiction because $\beta \in I_n$. This contradiction implies that $k \leq \gamma_1$. 
  So we have that $\gamma_1 > h_p$, $\gamma_0 < \gamma_1$, $\gamma_0 \in I_{n}$ and $\gamma_1 \in \M$.
  Also, for each $j' \in \{j \in \M : \gamma_0 < j \leq \gamma_1\}$, we have deleted $j'$ in a Type Two reduction, where we have performed the reduction on $j', j$ for some $\gamma_0 \leq j \leq \gamma_1$. So we have satisfied all conditions to use \lref{l:tech} to conclude that $\gamma_0 > h_p$, $\gamma_0 \in \M$, and $\{j \in I_{n} : -2\gamma_1 \preccurlyeq a_j \prec -2\gamma_0\} = \emptyset$. Using an analogous argument and \lref{l:tech2} we can conclude that $\gamma_3 \in \M$, and $\{j \in I_{n} : -2\gamma_3 \prec a_j \preccurlyeq -2\gamma_2\} = \emptyset$. Similar arguments show the same result when $\gamma_1 < \gamma_2 \leq h_p$. 
Therefore,
  \begin{align*}
    X(I_{n}, \gamma_0, \gamma_3) &= \{j \in I_{n} : \beta-1=-2\gamma_2 \prec a_j \prec -2\gamma_1=\beta+\alpha\}
    \subseteq \{\beta, \beta+1, \ldots, \beta+\alpha-1\},
  \end{align*}
  as $\{\beta, \beta+1, \ldots, \beta+\alpha-1\}$ is a run of $\N$-elements. But from \lref{l:mroprop} we know that the only element in this run which has not been deleted is $\beta$. So $X(I_{n}, \gamma_0, \gamma_3) = \{\beta\}$ and we are done.
\end{proof}

\begin{ex}In Case 1, when $p = 11$ the matrix $\C_2$ is
  \[
  \begin{pmatrix}
    0&0&0&1\\
    0&0&1&1\\
    0&1&0&1\\
    1&1&1&0\\
  \end{pmatrix}.
  \]
\end{ex}

Before describing the final stage of the reduction operations we will need to prove that particular elements of $\Z_p^*$ are not deleted in Algorithm $2$.

\begin{lem}\label{l:mrtprop}
  We have that $\kappa_p, p-\kappa_p \in I_{t_2}$. In Case $2$ and Case $3$ we also have that $(p+\kappa_p)/2 \in I_{t_2}$.
\end{lem}
\begin{proof}
  We proved that $\kappa_p$ and $p-\kappa_p$ are never deleted in Algorithm $2$ in the proof of \lref{l:mrtred}. We will now show that in Case $2$ and Case $3$, $(p+\kappa_p)/2$ is also not deleted during Algorithm $2$. We first note that $(p+\kappa_p)/2 \equiv 2^{-1}\kappa_p \bmod p$. As $2^{-1} \in \Rp_p$ and $\kappa_p \in \M$ it follows that $(p+\kappa_p)/2 \in \M$ and so we know that $(p+\kappa_p)/2$ was not deleted in Algorithm $1$. So if $(p+\kappa_p)/2$ has been deleted then it has been deleted by some reduction $R_m$ in Algorithm $2$. In that reduction
  we must have had either $\gamma_0 = (p+\kappa_p)/2$ or $\gamma_3 = (p+\kappa_p)/2$. We have shown in \lref{l:mrtred} that we always have $\gamma_0, \gamma_3 \in \M$ and either $\gamma_0 < \gamma_3 \leq h_p$ or $\gamma_3 > \gamma_0 > h_p$. If $\gamma_3 = (p+\kappa_p)/2 > h_p$ then we must have $\gamma_0 > h_p$. So we can write $\gamma_0 = (p+k)/2$ for some odd $k < \kappa_p$. This is impossible in Case $2$ as $\kappa_p = 1$. In Case $3$ we know from the definition of $\kappa_p$ that $k$ must be an element of $\N$. So as $2^{-1} \in \Rp_p$ it follows that $\gamma_0 \equiv 2^{-1}k \in \N$, a contradiction. Hence we must have $\gamma_0 = (p+\kappa_p)/2$. Then as we delete $\gamma_0$ we must have $\beta < \gamma_0$. Also, as $-2((p+\kappa_p)/2) \equiv p-\kappa_p\bmod p$ and $\beta \in X(I_m, \gamma_0, \gamma_3)$ we must have had
  \[
  -2\gamma_3 \prec \beta \prec p-\kappa_p.
  \]
  As $\kappa_p < p/5$ it follows that $(p+\kappa_p)/2 < p-2\kappa_p \equiv -2\kappa_p\bmod p$. Now, from $\beta < (p+\kappa_p)/2$ it follows that
  \[
  -2\gamma_3 \prec \beta \prec -2\kappa_p \prec p-\kappa_p.
  \]
  So as $-2\kappa_p = a_{\kappa_p}$ and $X(I_m, \gamma_0, \gamma_3) = \{\beta\}$ it follows that $\kappa_p$ must have already been deleted, which is a contradiction. Hence $(p+\kappa_p)/2$ is never deleted, as claimed.
\end{proof}

As mentioned before \lref{l:kdel} we know that at the end of the iteration of the outer \textbf{for} loop of Algorithm $2$ with $\alpha = k$ we have that $\{z \in (I_\zeta \cap \N) \setminus \{p-1\} : l_z \leq k\} = \emptyset$. So it follows that $I_{t_2} \setminus \{p-1\} \subseteq \M$. Furthermore, in Case $2$ and Case $3$ we have $I_{t_2} \subseteq \M$.

\begin{algorithm}[H]
  \DontPrintSemicolon
  \SetKwInOut{Input}{input}\SetKwInOut{Output}{output}
  \Input{$t_2$, $I_{t_2}$}
  $\zeta := t_2$ \;
  \While{$|I_\zeta| > 2$}{
    Let $\delta$ be the element adjacent to $\kappa_p$ in $I_\zeta$ \;
    Let $\omega$ be such that $X(I_\zeta, \kappa_p, \delta) = \{\omega\}$ \;
    $I_{\zeta+1} := I_\zeta \setminus \{\delta, \omega\}$ \;
    $\zeta := \zeta +1$
  }
  \Return $\zeta$, $I_\zeta$\;
  \caption{Deleting all but $2$ remaining elements}
\end{algorithm}

There are many things which must be verified in order for Algorithm $3$ to be well defined. Let $t_3$ denote the number of reduction operations performed in Algorithm $1$, Algorithm $2$ and Algorithm $3$ combined. In all $3$ algorithms $I_{\zeta+1}$ is obtained by deleting exactly $2$ elements from $I_\zeta$. So as we initially had $I_0 = \Z_p^*$ it follows that $|I_\zeta| \equiv |\Z_p^*| \equiv 0 \bmod 2$ at any stage of Algorithm $3$.
From \lref{l:mroprop} and \lref{l:mrtprop} we know that $\kappa_p=\min(I_{t_2})$. So as $|I_\zeta|\ge2$ and $\kappa_p\notin\{\delta,\omega\}$ throughout Algorithm 3, it follows that $\kappa_p=\min(I_\zeta)$ for all $\zeta\ge t_2$ and $\delta$ is well defined. We next show that $\omega$ is well defined.

\begin{lem}\label{l:3term}
  In every iteration of the {\rm \textbf{while}} loop of Algorithm $3$ we have that $\delta \leq h_p$ and $|X(I_\zeta, \kappa_p, \delta)| = 1$ and 
  hence $I_{\zeta+1}$ is obtained by performing a Type Two reduction of $I_\zeta$.
\end{lem}

\begin{proof}
  Suppose, for a contradiction, that at some point of the algorithm we had either $\delta > h_p$ or $|X(I_\zeta, \kappa_p, \delta)| \neq 1$. Let $\zeta=m$ in the iteration of the \textbf{while} loop when this first occurs. We will first suppose that $\delta > h_p$ in this iteration.
  Note that from the discussion before Algorithm $3$ we know that $I_m \subseteq \M$ in Case $2$ and Case $3$, and that $p-1$ is the only possible element in $I_m \cap \N$ in Case $1$. If we are in Case $1$ and $\delta = p-1$ then from the definition of $\delta$ we must have $|I_m| = 2$. But this is a contradiction because the algorithm would have terminated, so it follows that $\delta \in \M$ in every case. As $|I_m| > 2$ we can find $y = \min(I_m \setminus \{\delta, \kappa_p\})$. Then we know from the definition of $\delta$ that $y > \delta > h_p$. By \lref{l:eqrowss} we know there exists $z \in X(I_m, \delta, y)$. If $z > h_p$ then 
  from $z \in X(I_m, \delta, y)$ we find that $2p - 2y < 2p - 2z < 2p - 2\delta$, and hence $\delta < z < y$, contradicting our choice of $y$. Hence we must have $z \leq h_p<\delta$ which means that $z = \kappa_p$ and so $\kappa_p \in X(I_m, \delta, y)$. As $|I_m| \equiv 0 \bmod 2$ it follows that there exists $x = \min(I_m \setminus \{\kappa_p, \delta, y\})$. Then we note that $x$ and $y$ are adjacent in $I_m$. If $a \in X(I_m, x, y)$ then, by using similar arguments as just used to show that $z=\kappa_p$, we can show that $a = \kappa_p$. But this is impossible since it would imply that 
  \[
  -2x \prec p-2\kappa_p \prec -2y \prec p-2\kappa_p \prec -2\delta.
  \]
  Hence $X(I_m, x, y) = \emptyset$. This contradiction of \lref{l:eqrowss} shows that $\delta \leq h_p$, as claimed. So it must be that $|X(I_\zeta, \kappa_p, \delta)| \neq 1$ in the iteration of the \textbf{while} loop with $\zeta=m$. 
  By \lref{l:eqrowss}, we know that 
  $|X(I_m, \kappa_p, \delta)| \geq 3$. Suppose there is some element $z \in X(I_m, \kappa_p, \delta)$ with $z \leq h_p$. 
  Then given that $\kappa_p<\delta\le h_p$,
  we have that $p-2\delta < p-2z < p-2\kappa_p$, and hence $\kappa_p < z < \delta$, which contradicts the definition of $\delta$. Thus $z > h_p$ for all $z \in X(I_m, \kappa_p, \delta)$. Let $a, b \in X(I_m, \kappa_p, \delta)$ with $a < b$ such that $2b - 2a = \min\{2d - 2c : c, d \in X(I_m, \kappa_p, \delta) \text{ with } c < d\}$. Suppose that $a$ and $b$ are not adjacent, so there exists some $c \in I_m$ with $a < c < b$. Then $2p - 2b < 2p - 2c < 2p - 2a$, 
  which implies that
  \[
  -2\delta \prec -2b \prec -2c \prec -2a \prec -2\kappa_p.
  \]
  So $c \in X(I_m, \kappa_p, \delta)$ and $2b - 2c< 2b - 2a$, contradicting the definition of $a$ and $b$ and implying that $a$ and $b$ are adjacent in $I_m$. From \lref{l:eqrowss} we know there exists $d \in X(I_m, a, b) \subseteq X(I_m, \kappa_p, \delta)$. As $d \in X(I_m, \kappa_p, \delta)$ we know that $d > h_p$. Hence we have 
  $2p - 2b < 2p - 2d < 2p - 2a$, and thus $a < d < b$. This is a contradiction as $a$ and $b$ are adjacent. Hence, $\delta \leq h_p$ and $|X(I_\zeta, \kappa_p, \delta)| = 1$ in every iteration of the \textbf{while} loop.
\end{proof}

We next show that in all $3$ cases, 
\begin{equation}\label{e:endgame}
  \C_3=\begin{pmatrix}
    0 & 1 \\
    1 & 0
  \end{pmatrix}.
  \end{equation}

\begin{lem}\label{l:mrtout}
  In Case $1$ we have that $I_{t_3} = \{1, p-1\}$. In Case $2$ and Case $3$ we have $I_{t_3} = \{\kappa_p, (p+\kappa_p)/2\}$.
\end{lem}

\begin{proof}
  It is clear from Algorithm 3 that $\kappa_p\in I_{t_3}$. Also,
  since $|I_0|\equiv0\bmod2$ and each reduction removes two elements from $I$, it is clear that $|I_{t_3}| = 2$.
  
  First consider Case 1. Suppose, for a contradiction, that $p-1$ gets deleted in reduction $R_{m+1}$. 
  In this reduction we must have $p-1\in\{\delta,\omega\}$.
  As $\delta \leq h_p$ from \lref{l:3term}, it follows that $\omega = p-1 \in X(I_m, 1, \delta)$. However, this would mean that $-2\delta \prec -1 \prec -2$, which is absurd.
  
  Next consider Case $2$ and Case $3$. 
  Suppose, for a contradiction, that $(p+\kappa_p)/2 \not\in I_{t_3}$. From \lref{l:mrtprop} we know that $(p+\kappa_p)/2\in I_{t_2}$. Hence we must have deleted $(p+\kappa_p)/2$ in some reduction $R_{m+1}$ performed during Algorithm $3$. 
  In that reduction we must have $(p+\kappa_p)/2\in\{\delta,\omega\}$.
  As $\delta \leq h_p$ from \lref{l:3term} it follows that $\omega = (p+\kappa_p)/2$. Hence $-2\omega \equiv p-\kappa_p \bmod p$. So as $\omega \in X(I_m, \kappa_p, \delta)$, it follows that $-2\delta \prec p-\kappa_p \prec -2\kappa_p$, which is false. This contradiction shows that $I_{t_3} = \{\kappa_p, (p+\kappa_p)/2\}$, as claimed.
\end{proof}

\begin{thm}\label{t:fullcyc}
  The matrix $\C$ is non-singular over $\Z_2$ in each case.
\end{thm}

\begin{proof}
  Using \lref{l:mrtout} and the definition of $\C$ it is easy to verify that
  \eref{e:endgame} holds.
  By \lref{l:mrocor}, \lref{l:mrtred} and \lref{l:3term} we know that $I_{t_3}$ has been obtained from $\Z_p^*$ by a sequence of Type One and Type Two reductions, so the result follows from \lref{l:detpres}. 
\end{proof}

Combining \lref{l:matrh} with \tref{t:fullcyc} and recalling that $\mathscr{L}_3$ is atomic we have shown the following.

\begin{thm}\label{t:3mod8rh}
  The Latin square $\mathscr{L}_p$ is row-Hamiltonian for any prime $p \equiv 1 \bmod 8$ or $p \equiv 3 \bmod 8$.
\end{thm}

\subsection{Proving that $\nu(\mathscr{L}_p) = 4$}\label{ss:nu4}

In \sref{ss:conj} we showed that $\mathscr{L}_p$ is row-Hamiltonian for all primes $p \equiv 1 \bmod 8$ or $p \equiv 3 \bmod 8$. So we know that $\nu(\mathscr{L}_p) \in \{2, 4, 6\}$, and in this subsection we will prove that $\nu(\mathscr{L}_p) = 4$, except when $p \in \{3, 19\}$. For a Latin square $L$ with symbol set $S$, we denote the symbol permutation between symbols $i, j \in S$ by $s_{i, j}$. Each cycle in the disjoint cycle representation of $s_{i, j}$ is a \emph{symbol cycle}. Similarly, the column permutation between columns $i$ and $j$ is denoted by $c_{i, j}$ and every cycle in $c_{i, j}$ is a \emph{column cycle}. First, we will show that our Latin squares are not atomic, by showing the existence of a symbol permutation of $\mathscr{L}_p$ which is not a $p$-cycle. This will demonstrate that the $(3, 2, 1)$-conjugate of $\mathscr{L}_p$ has a row permutation which is not a $p$-cycle. We will need some results regarding characters of finite fields. Let $\mathbb{T}$ denote the group of complex numbers of absolute value $1$ and let $\syms[x]$ denote the ring of polynomials over $\syms$. We will extend any multiplicative character $\chi : \Z_p^* \to \mathbb{T}$ to $\syms$ by defining $\chi(0) = 0$. A particular character which is useful to us is the quadratic character, which we will denote by $\mu$. The following theorem~\cite{MR1429394} is a version of the well known Weil bound.

\begin{thm}\label{t:Weil}
  Let $\chi$ be a character of $\Z_p^*$ of order $m > 1$, and let
  $f\in \Z_p[x]$ be a polynomial of degree $d > 0$ which has $d$
  distinct roots in $\Z_p$. Then,
  \begin{equation}\label{e:weil}
    \left\vert \displaystyle\sum_{c\smash{\in \Z_p}} \chi\big(f(c)\big) \right\vert \leq (d - 1)p^{1/2}.
  \end{equation}
\end{thm}

In the case where $\chi = \mu$ and $f$ is a quadratic polynomial with non-zero discriminant, a result from~\cite{MR1429394} gives an explicit value for the sum in \eref{e:weil}.

\begin{thm}
  \label{t:quadWeil}
  Let $f(x) = a_2x^2 + a_1x + a_0\in \Z_p[x]$, with $a_2 \not\equiv 0$, and $a_1^2 - 4a_0a_2 \not\equiv 0\bmod p$. Then,
  \begin{equation*}
    \displaystyle\sum_{c \in \Z_p} \mu\big(f(c)\big) = -\mu(a_2).
  \end{equation*}
\end{thm}

We are now ready to prove that $\mathscr{L}_p$ is not atomic for
sufficiently large $p$. 

\begin{lem}\label{l:nunot6}
  Let $p \geq 1697$ be a prime with $p \equiv 1 \bmod 8$ or
  $p \equiv 3 \bmod 8$. Then $\mathscr{L}_p$ is not atomic.
\end{lem}

\begin{proof}
  Let $x \in \Z_p$ and suppose that $\{x, x+1, x+2^{-1}\} \subseteq \Rp_p$
  and $\{-2^{-1}x-3\cdot4^{-1}, 4^{-1} - 2^{-1}x\} \subseteq\Np_p$.
  Then the following $6$ triples are in $\mathscr{L}_p$:
  \[
  \begin{array}{lllll}
    (x+1,2x+1,1),
    &&(x+3\cdot2^{-1},2x+2,1),
    &&(x+2^{-1},2^{-1}x+3\cdot4^{-1},1),\\
    (x+1,2x+2,0),
    &&(x+3\cdot2^{-1},2^{-1}x+3\cdot4^{-1},0),
    &&(x+2^{-1},2x+1,0).
  \end{array}
  \]
  Hence the symbol permutation $s_{0,1}$ contains the $3$-cycle $(x+1, x+3\cdot2^{-1}, x+2^{-1})$, and thus the row permutation $r_{0, 1}$ of the $(3, 2, 1)$-conjugate of $\mathscr{L}_p$ is not a $p$-cycle.
  So to prove the claim, it suffices to use the Weil bound to show that such an $x \in \Z_p$ exists. Define
  \[
  A = \{x \in \Z_p : x, x+1, x+2^{-1} \in \Rp_p \text{ and } -2^{-1}x-3\cdot4^{-1}, 4^{-1} - 2^{-1}x \in \Np_p\}.
  \]
  It suffices to show that $A \neq \emptyset$. Define
  \[
  Q(x) = \big(1 + \mu(x)\big)\big(1 + \mu(x+1)\big)\big(1 + \mu(x+2^{-1})\big)\big(1 - \mu(-2^{-1}x-3\cdot4^{-1})\big)\big(1 - \mu(4^{-1} - 2^{-1}x)\big).
  \]
  If $x \in A$ then $Q(x) = 32$. If $x \in \Z_p$ such that $0 \in \{x, x+1, x+2^{-1}, -2^{-1}x-3\cdot4^{-1}, 4^{-1} - 2^{-1}x\}$ then $Q(x) \leq 16$. There are at most $5$ values for $x$ which satisfy this condition. For all other $x \in \Z_p \setminus A$ we know that $Q(x) = 0$. So, it follows that
  \begin{equation}\label{e:s}
    \sum_{x \in \Z_p} Q(x) \leq 32|A| + 80.
  \end{equation}
  Expanding $Q(x)$ and noting that $\mu$ is a homomorphism on $\Z_p^*$, we see that $Q(x)$ is a sum of terms $\pm\mu\big(K(x)\big)$ where $K$ is a product of $k$ distinct factors in $\{x, x+1, x+2^{-1}, -2^{-1}x-3\cdot4^{-1}, 4^{-1} - 2^{-1}x\}$, for some $k \in \{0, 1, 2, 3, 4, 5\}$. This sum contains one constant term $K = 1$, $5$ linear terms, $10$ quadratic terms, $10$ cubic terms, $5$ quartic terms and a quintic term. In each case we know that $K$ is the product of distinct linear factors and so both \tref{t:quadWeil} and \tref{t:Weil} apply, giving
  \[
  \sum_{x \in \Z_p} Q(x) \geq p - 39\sqrt{p} - 10.
  \]
  Combining with \eref{e:s}, we see that
  \[
  p - 39\sqrt{p} - 10 \leq 32|A| + 80.
  \]
  But $p - 39\sqrt{p} - 90 > 0$ for $p \geq 1697$. Hence we must have $|A| \geq 1$ and the result follows.
\end{proof}

In the proof of \lref{l:nunot6}, it is clear that if $A\neq\emptyset$ then the symbol permutation $s_{0, 1}$ of $\mathscr{L}_p$ is not a $p$-cycle, except when $p=3$. Computationally it has been verified that $A\neq\emptyset$ for all primes $p \equiv 1 \bmod 8$ or $p \equiv 3 \bmod 8$ with $20 \leq p \leq 1696$ and for $p=11$. It has also been directly computed that $\nu(\mathscr{L}_{17}) = 4$. This leads to the following result.

\begin{cor}\label{c:nunot6}
  The Latin square $\mathscr{L}_p$ is not atomic for prime $p \equiv 1 \bmod 8$ or $p \equiv 3 \bmod 8$, except when $p \in \{3, 19\}$.
\end{cor}

We will now strengthen \cyref{c:nunot6} by proving that $\mathscr{L}_p$ contains no symbol permutation which is a $p$-cycle when $p \notin \{3, 19\}$. This fact will not be used to prove that $\nu(\mathscr{L}_p) = 4$, but will be needed in \sref{s:app}. By \lref{l:sufconrh} we have already proven the claim in the case when $p \equiv 3 \bmod 8$. To prove the case where $p \equiv 1 \bmod 8$ we must prove that the symbol permutation $s_{0, c}$ of $\mathscr{L}_p$ is not a $p$-cycle for some $c \in \Np_p$.

\begin{lem}\label{l:largep}
  Let $p \equiv 1 \bmod 8$ be prime with $p \geq 1697$ and let $c \in \Np_p$. The symbol permutation $s_{0, c}$ of $\mathscr{L}_p$ is not a $p$-cycle.
\end{lem}
\begin{proof}
  Suppose that there exists some $x \in \Z_p$ such that $\{x, x-c\} \subseteq \Rp_p$ and $\{c-2x, 2^{-1}c-2x, 3\cdot2^{-1}c-2x\} \in \Np_p$. Then the following $6$ triples are in $\mathscr{L}_p$:
  \[
  \begin{array}{lllll}
    (x,2x,0),
    &&(4x-2c,2x-c,0),
    &&(4x-c,2x-2c^{-1},0),\\
    (x,2x-c,c),
    &&(4x-2c,2x-2c^{-1},c),
    &&(4x-c,2x,c).
  \end{array}
  \]
  Thus if such an $x \in \Z_p$ exists then the symbol permutation $s_{0, c}$ is not a $p$-cycle. The lemma now follows by arguing along the same lines as in the proof of \lref{l:nunot6}.
\end{proof}

From the proof of \lref{l:largep} it follows that if there exists some $c \in \Np_p$ and $x \in \Rp_p$ such that $x-c \in \Rp_p$ and $c-2x, 2^{-1}c-2x, 3\cdot2^{-1}c-2x \in \Np_p$ then the symbol permutation $s_{0, c}$ of $\mathscr{L}_p$ is not a $p$-cycle when $p \equiv 1 \bmod 8$. Computationally it has been verified that such elements exist for all primes $p \equiv 1 \bmod 8$ with $41 \leq p \leq 1657$. Also, the symbol permutation $s_{0, 3}$ of $\mathscr{L}_{17}$ contains the $3$-cycle $(1, 11, 8)$. In combination with the discussion before \lref{l:largep}, we get the following result.

\begin{cor}\label{c:symbperm}
  No symbol cycle of $\mathscr{L}_p$ is a $p$-cycle for any $p \equiv 1 \bmod 8$ or $p \equiv 3 \bmod 8$ except when $p \in \{3, 19\}$.
\end{cor}

By combining \tref{t:3mod8rh} and \cyref{c:nunot6} we know that $\nu(\mathscr{L}_p) \in \{2, 4\}$ when $p \not\in \{3, 19\}$. To show that $\nu(\mathscr{L}_p) = 4$ we will need the following lemma, from~\cite{MR2134185}.

\begin{lem}\label{l:213conj}
  Let $p$ be prime and let $a, b \in \Z_p$ such that $ab, (1-a)(1-b) \in \Rp_p$. Then,
  \begin{enumerate}[(i)]
  \item If $p \equiv 1 \bmod 4$ then the $(2, 1, 3)$-conjugate of the Latin square $\mathcal{L}[a, b]$ is the Latin square $\mathcal{L}[1-a, 1-b]$.
  \item If $p \equiv 3 \bmod 4$ then the $(2, 1, 3)$-conjugate of the Latin square $\mathcal{L}[a, b]$ is the Latin square $\mathcal{L}[1-b, 1-a]$.
  \end{enumerate}
\end{lem}

We can now prove \tref{t:RCnotS}.

\begin{proof}[Proof of \tref{t:RCnotS}]
The $19 \times 19$ Latin square $\mathcal{L}_2[3, 14, 8, 7, 11, 14]$
has $\nu = 4$, so we may assume that $p \neq 19$.
\tref{t:3mod8rh} tells us that $\mathscr{L}_p$ is row-Hamiltonian. 
As $\mathscr{L}_p=\mathcal{L}[-1, 2]$,
we know from \lref{l:213conj} that $\mathscr{L}_p$ is equal to its
$(2,1,3)$-conjugate when $p \equiv 3 \bmod 8$ and is isomorphic to
its $(2,1,3)$-conjugate when $p \equiv 1 \bmod 8$. In either case, the
$(2,1,3)$-conjugate of $\mathscr{L}_p$ must be row-Hamiltonian. Combining
with \cyref{c:nunot6} and \lref{l:nueven}, we see that
$\nu(\mathscr{L}_p) = 4$.
\end{proof}

\subsection{On Latin squares with $\nu = 4$}\label{ss:lsnu4}

As mentioned in \sref{s:bm} there are no Latin squares with $\nu = 4$
of order less than $11$. In particular there is no $3 \times 3$ Latin
square with $\nu = 4$ and so combining this with \tref{t:RCnotS} we
have resolved the question of existence of Latin squares with $\nu=4$
of prime order $p \equiv 1 \bmod 8$ or $p \equiv 3 \bmod 8$. We note
that the Latin squares $\mathscr{L}_p$ are all of the form
$\mathcal{L}[a, 1-a]$ for some $a \in \Z_p$. There are several reasons
why Latin squares of this form are promising candidates to have $\nu=4$.
Firstly, we only need the product $a(1-a) \in \Rp_p$ for the Latin
square $\mathcal{L}[a, 1-a]$ to be well defined. Furthermore, we know
from \lref{l:213conj} that $\mathcal{L}[a, 1-a]$ is isomorphic to its
$(2, 1, 3)$-conjugate. So it follows that $\nu(\mathcal{L}[a, 1-a])
\in \{0, 4, 6\}$. So it is perhaps slightly more likely that a square of
this form will have $\nu = 4$. In addition to the family shown in
\tref{t:RCnotS}, there are examples
with $\nu = 4$ of order $p \equiv 5 \bmod 8$ or $p \equiv 7\bmod 8$.
For example, the $29 \times 29$ Latin square $\mathcal{L}[3,27]$
has $\nu = 4$. This is the least prime order $p\equiv 5 \bmod 8$
for any $\mathcal{L}[a, 1-a]$ with $\nu=4$.
For all primes $p \equiv 5 \bmod 8$ with $29 \leq p \leq 1000$,
there exists a Latin square $\mathcal{L}[a, 1-a]$ with $\nu = 4$ for
some $a \in \Z_p$, except when $p \in \{37, 317, 661\}$. However, there
are Latin squares generated by quadratic orthomorphisms of these
orders with $\nu = 4$, including the $37 \times 37$ Latin square
$\mathcal{L}[9, 25]$, the $317 \times 317$ Latin square
$\mathcal{L}[3, 13]$ and the $661 \times 661$ Latin square
$\mathcal{L}[42, 532]$. The $47 \times 47$ Latin square
$\mathcal{L}[6, 42]$ also has $\nu = 4$, and is the Latin square of
least order $p \equiv 7 \bmod 8$ of the form $\mathcal{L}[a, 1-a]$
with $\nu = 4$. For all primes $p \equiv 7 \bmod 8$ with $47 \leq p
\leq 1000$, there exists a Latin square $\mathcal{L}[a, 1-a]$ with
$\nu = 4$ for some $a \in \Z_p$. These examples raise the tantalising
prospect of extending \tref{t:RCnotS} to all large prime orders.

For composite orders, not much is known about row-Hamiltonian Latin
squares and even less is known about Latin squares with $\nu=4$.  We
know that any such square must have odd order at least $15$. With a
computer it is easy to check that there is no diagonally cyclic Latin
square of order $15$ with $\nu = 4$. So the smallest known example of
a Latin square of composite order with $\nu = 4$ is the $21 \times 21$
diagonally cyclic Latin square, whose first row is $(0, 8, 15, 18, 13,
11, 14, 4, 12, 2, 9, 7, 17, 3, 5, 10, 19, 6, 16, 20, 1)$.

\section{Falconer varieties}\label{s:app}

In this section we will prove \tref{t:falcsol}. First we give the universal algebraic definition of a quasigroup.

\begin{defin}
  A \emph{quasigroup} $(Q, \cdot, \backslash, /)$ is a non-empty set $Q$ with binary operations $(\cdot)$, $(\backslash)$ and $(/)$, called multiplication, left division and right division respectively, such that the following properties hold for all $x, y \in Q$:
  \begin{enumerate}[(i)]
  \item $x \cdot (x \backslash y) = y$,
  \item $x \backslash (x \cdot y) = y$,
  \item $(y / x) \cdot x = y$,
  \item $(y \cdot x) / x = y$.
  \end{enumerate}
\end{defin} 

We will often refer to the quasigroup $(Q, \cdot, \backslash, /)$ simply by $Q$. By relabelling the elements of a finite quasigroup $Q$ we may assume that $Q = \Z_n$ for some positive integer $n$, which we call the order of $Q$. From now on we will assume that every finite quasigroup is of this form. Recall that the multiplication table of a quasigroup is a Latin square. Conversely, every Latin square is the multiplication table of some finite quasigroup. For an $n \times n$ Latin square $L$, we can define a quasigroup $Q = \Z_n$ by the following operations. We define the operation $(\cdot)$ on $Q$ by $a \cdot b = L_{a, b}$. We define left and right division on $Q$ as follows. Let $a, b \in Q$. Let $y \in Q$ satisfy $L_{a, y} = b$ and $x \in Q$ be such that $L_{x, b} = a$. We then define $a \backslash b = y$ and $a / b = x$. Then $Q$ is a quasigroup with multiplication table $L$.

We will now construct an infinite family of Falconer varieties. Let $Q$ be a quasigroup. For $x \in Q$ we define $L_x : Q \to Q$ by $L_x(y) = x \cdot y$. Then $L_x$ is bijective, with inverse $L_x^{-1} : Q \to Q$ defined by $L_x^{-1}(y) = x \backslash y$. We also define $R_x : Q \to Q$ by $R_x(y) = y \cdot x$. Then $R_x^{-1}$ is also bijective with inverse $R_x^{-1} : Q \to Q$ defined by $y \mapsto y / x$. Let $p > 3$ be prime with $p \equiv 1 \bmod 8$ or $p \equiv 3 \bmod 8$. Define $m_p = \lcm(1, 2, \ldots, p-1)$, and let $E_p$ be the set consisting of the following two identities:
\[
(L_x^{-1}L_y)^p(z) = z,
\]
\[
(R_x^{-1}R_y)^{m_p}(z) = z.
\]
Let $\V_p$ denote the loop variety defined by $E_p$. We will show that each $\V_p$ is a Falconer variety, proving \tref{t:falcsol}.

\begin{proof}[Proof of \tref{t:falcsol}]
  We first show that $\V_p$ is non-trivial. Suppose that $p \neq 19$. Let $Q$ be a loop that is isotopic to the $(2, 3, 1)$-conjugate of $\mathscr{L}_p$. Let $S$ denote the multiplication table of $Q$. 
  We will show that $Q$ satisfies the identities in $E_p$. Let $x, y \in Q$, and let $r$ denote the row permutation $r_{x, y}$ of $S$. We will show that $L_x^{-1}L_y = r$. Let $z \in Q$ and note that $L_x^{-1}L_y(z) = y \cdot (x \backslash z) = S_{y, w}$ where $S_{x, w} = z$. By definition of $r$, we have that $r(z) = r(S_{x, w}) = S_{y, w}$, proving the claim that $L_x^{-1}L_y = r$. From the proof of \tref{t:RCnotS} we know that the $S$ is row-Hamiltonian. So $r$ is a $p$-cycle and thus $r^p$ is the identity permutation and so $(L_x^{-1}L_y)^p(z) = z$ for all $z \in Q$ and the first identity in $E_p$ is satisfied. Let $c=R_x^{-1}R_y$ denote the column permutation $c_{x, y}$ of $S$. We know from \cyref{c:symbperm} that no symbol permutation of $\mathscr{L}_p$ is a $p$-cycle. 
  It follows that no column permutation of $S$ is a $p$-cycle and hence $c^{m_p}$ is also the identity permutation. Thus the second identity in $E_p$ is satisfied. So $Q \in \V_p$, meaning that $\V_p$ is non-trivial.
  We can use similar arguments to deal with the case when $p=19$. It suffices to note that the $19 \times 19$ Latin square $\mathcal{L}[3, 13]$ is row-Hamiltonian, but has no column cycle which is a $19$-cycle. Hence $\V_{19}$ is non-trivial also.
  
  The fact that $\V_p$ is isotopically $L$-closed is due to the fact that isotopy preserves the lengths of row cycles and column cycles.
  
  We will now show that $\mathscr{V}_p$ is anti-associative. Suppose that $G \in \mathscr{V}_p$ is a group and let $x, y \in G$. Then $z = (L_x^{-1}L_y)^p(z) = (y \cdot x^{-1})^p \cdot z$ for all $z \in G$, meaning the order of $y \cdot x^{-1}$ divides $p$. As this is true for any $x, y \in G$ it follows that every element of $G$ has order dividing $p$. Similarly, $z = (R_x^{-1}R_y)^{m_p}(z) = z \cdot (x^{-1} \cdot y)^{m_p}$ for all $z \in G$, and so the order of every element in $G$ divides $m_p = \lcm(1, 2, \ldots, p-1)$, which is relatively prime to $p$. Hence $G$ contains only the identity element and $\V_p$ is anti-associative. 
  
  Finally, we will show that the smallest, non-trivial member of $\V_p$ has order $p$. We have seen that if a finite, non-trivial loop $Q$ satisfies the first identity in $E_p$ then every row permutation of the multiplication table of $Q$ has order $p$. Hence the order of $Q$ must be a multiple of $p$, and the result follows.
\end{proof}

\section{Conclusion and further research}\label{s:c}

We have constructed the first infinite family of Latin squares with $\nu=4$. We have also used these Latin squares to construct infinitely many non-trivial, anti-associative, isotopically $L$-closed loop varieties, solving an open problem published by Falconer.

Rosa~\cite{MR3954017} conjectured that the number of (isomorphism classes of) perfect $1$-factorisations of $K_{2n}$ approaches infinity as $n \to \infty$. Combining this with the fact that a perfect $1$-factorisation of $K_{2n-1, 2n-1}$ can be constructed from a perfect $1$-factorisation of $K_{2n}$, and with the equivalence of perfect $1$-factorisations of $K_{n, n}$ and $n \times n$ row-Hamiltonian Latin squares, we obtain the conjecture that the number of (isotopism classes of) row-Hamiltonian Latin squares of odd order tends to infinity. This conjecture has been resolved only in the case of Latin squares of order $p^2$ for odd primes $p$. In particular, it was shown in~\cite{MR1899629} that the number of isotopism classes of $n \times n$ row-Hamiltonian Latin squares increases at least linearly with $n$ when $n$ is the square of an odd prime. Call a pair $(a, b) \in \Z_p^* \times \Z_p^*$ with $a < b$ \emph{valid} if the $p \times p$ Latin square $\mathcal{L}[a, b]$ is row-Hamiltonian. The following plot shows the number of valid pairs in $\Z_p^*\times\Z_p^*$.

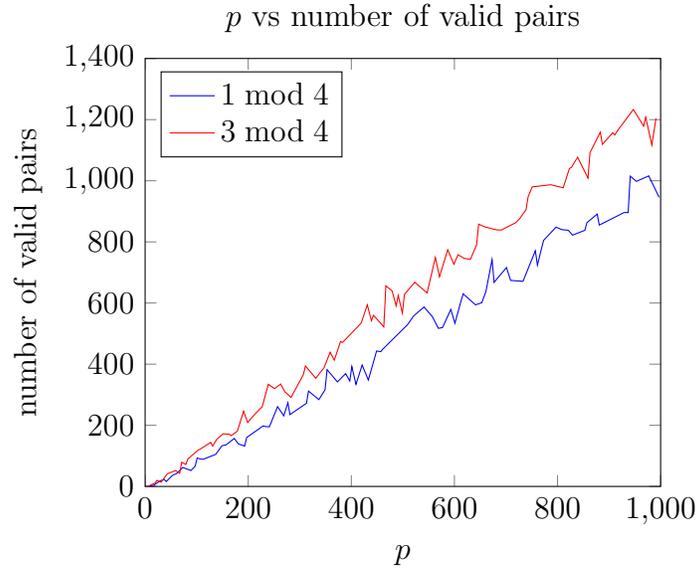
\begin{center}
  \begin{tikzpicture}
    \begin{axis}[title = {$p$ vs number of valid pairs}, xlabel = {$p$}, ylabel = {number of valid pairs}, xmin=0, xmax=1000, ymin=0, ymax=1400, xtick = {0, 200, 400, 600, 800, 1000}, ytick = {0, 200, 400, 600, 800, 1000, 1200, 1400}, legend pos = north west]
      \addplot[color = blue, mark=circle]
      coordinates {(5, 0)
	(13, 3)
	(17, 4)
	(29, 18)
	(37, 23)
	(41, 16)
	(53, 36)
	(61, 42)
	(73, 62)
	(89, 52)
	(97, 66)
	(101, 93)
	(109, 89)
	(113, 89)
	(137, 105)
	(149, 133)
	(157, 136)
	(173, 157)
	(181, 139)
	(193, 132)
	(197, 160)
	(229, 198)
	(233, 196)
	(241, 195)
	(257, 261)
	(269, 231)
	(277, 274)
	(281, 235)
	(293, 249)
	(313, 272)
	(317, 312)
	(337, 284)
	(349, 316)
	(353, 382)
	(373, 342)
	(389, 369)
	(397, 345)
	(401, 392)
	(409, 334)
	(421, 397)
	(433, 348)
	(449, 443)
	(457, 441)
	(461, 449)
	(509, 529)
	(521, 558)
	(541, 587)
	(557, 556)
	(569, 517)
	(577, 520)
	(593, 579)
	(601, 534)
	(613, 608)
	(617, 630)
	(641, 594)
	(653, 601)
	(661, 637)
	(673, 742)
	(677, 668)
	(701, 716)
	(709, 674)
	(733, 671)
	(757, 771)
	(761, 724)
	(769, 779)
	(773, 805)
	(797, 848)
	(809, 840)
	(821, 838)
	(829, 822)
	(853, 838)
	(857, 863)
	(877, 891)
	(881, 855)
	(929, 896)
	(937, 896)
	(941, 1015)
	(953, 998)
	(977, 1016)
	(997, 946)
      };
      \addplot[color = red, mark=circle]
      coordinates {
	(3, 0)
	(7, 0)
	(11, 5)
	(19, 10)
	(23, 20)
	(31, 14)
	(43, 41)
	(47, 44)
	(59, 52)
	(67, 43)
	(71, 79)
	(79, 72)
	(83, 90)
	(103, 119)
	(107, 122)
	(127, 144)
	(131, 132)
	(139, 154)
	(151, 172)
	(163, 171)
	(167, 166)
	(179, 181)
	(191, 246)
	(199, 209)
	(211, 233)
	(223, 254)
	(227, 260)
	(239, 334)
	(251, 320)
	(263, 335)
	(271, 310)
	(283, 291)
	(307, 366)
	(311, 394)
	(331, 354)
	(347, 388)
	(359, 439)
	(367, 413)
	(379, 474)
	(383, 471)
	(419, 534)
	(431, 594)
	(439, 542)
	(443, 560)
	(463, 522)
	(467, 656)
	(479, 640)
	(487, 591)
	(491, 625)
	(499, 567)
	(503, 628)
	(523, 668)
	(547, 633)
	(563, 749)
	(571, 686)
	(587, 774)
	(599, 727)
	(607, 758)
	(619, 746)
	(631, 743)
	(643, 790)
	(647, 858)
	(659, 849)
	(683, 839)
	(691, 839)
	(719, 863)
	(727, 877)
	(739, 906)
	(743, 948)
	(751, 980)
	(787, 987)
	(811, 977)
	(823, 1040)
	(827, 1044)
	(839, 1077)
	(859, 1009)
	(863, 1092)
	(883, 1159)
	(887, 1119)
	(907, 1157)
	(911, 1150)
	(919, 1170)
	(947, 1233)
	(967, 1179)
	(971, 1208)
	(983, 1117)
	(991, 1204)};
      \legend{$1 \bmod 4$, $3 \bmod 4$}
    \end{axis}
  \end{tikzpicture}
\end{center}

This suggests that the conjecture is true for Latin squares of prime
order.  Note that Dr\'apal and Wanless~\cite{DW23} have shown that for
prime orders $\mathcal{L}[a,b]$ is isomorphic to $\mathcal{L}[c, d]$
only if $\{a,b\}=\{c,d\}$. A similar statement is likely to hold with
isotopic in place of isomorphic.

In fact, this plot suggests that the number of isotopism classes of row-Hamiltonian Latin squares of order $p$ generated by quadratic orthomorphisms increases linearly with $p$. We would like to prove that this is indeed the case, potentially using \tref{t:linkmatcyc}, as we did to prove that the Latin squares $\mathscr{L}_p$ are row-Hamiltonian.

As mentioned in \sref{s:bm}, given an $n \times n$ row-Hamiltonian Latin square $L$, you can construct a perfect $1$-factorisation of $K_{n, n}$. It is known that if $L$ also satisfies certain symmetry conditions, you can construct a perfect $1$-factorisation of $K_{n+1}$ from $L$ (see~\cite{MR2130738}). When $p \equiv 3 \bmod 4$ and $a \in N_p \setminus \{-1\}$ the $p \times p$ Latin square $\mathcal{L}[a, a^{-1}]$ satisfies the required symmetry conditions~\cite{MR623318}. The following plot shows the number of row-Hamiltonian Latin squares $\mathcal{L}[a, a^{-1}]$ of prime order $p \equiv 3 \bmod 4$.

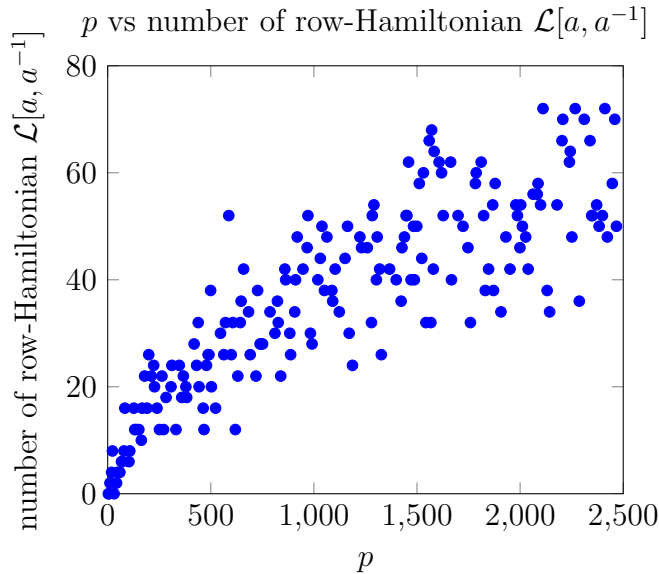
\begin{center}
	\begin{tikzpicture}
		\begin{axis}[title = {$p$ vs number of row-Hamiltonian $\mathcal{L}[a, a^{-1}]$}, xlabel = {$p$}, ylabel = {number of row-Hamiltonian $\mathcal{L}[a, a^{-1}]$}, xmin=0, xmax=2500, ymin=0, ymax=80, xtick = {0, 500, 1000, 1500, 2000, 2500}, ytick = {0, 20, 40, 60, 80}]
			\addplot[color = blue, only marks, mark=*]
			coordinates {(3,0)
				(7,0)
				(11,2)
				(19,4)
				(23,8)
				(31,0)
				(43,2)
				(47,4)
				(59,4)
				(67,6)
				(71,6)
				(79,8)
				(83,16)
				(103,6)
				(107,8)
				(127,16)
				(131,12)
				(139,12)
				(151,12)
				(163,10)
				(167,16)
				(179,22)
				(191,16)
				(199,26)
				(211,22)
				(223,24)
				(227,20)
				(239,16)
				(251,12)
				(263,22)
				(271,12)
				(283,18)
				(307,20)
				(311,24)
				(331,12)
				(347,24)
				(359,18)
				(367,22)
				(379,20)
				(383,18)
				(419,28)
				(431,24)
				(439,32)
				(443,20)
				(463,16)
				(467,12)
				(479,24)
				(487,26)
				(491,26)
				(499,38)
				(503,20)
				(523,16)
				(547,30)
				(563,26)
				(571,32)
				(587,52)
				(599,26)
				(607,32)
				(619,12)
				(631,22)
				(643,32)
				(647,36)
				(659,42)
				(683,34)
				(691,26)
				(719,22)
				(727,38)
				(739,28)
				(743,28)
				(751,28)
				(787,34)
				(811,30)
				(823,36)
				(827,32)
				(839,22)
				(859,42)
				(863,40)
				(883,30)
				(887,26)
				(907,34)
				(911,40)
				(919,48)
				(947,42)
				(967,46)
				(971,52)
				(983,30)
				(991,28)
				(1019,40)
				(1031,44)
				(1039,50)
				(1051,38)
				(1063,48)
				(1087,38)
				(1091,36)
				(1103,42)
				(1123,34)
				(1151,44)
				(1163,50)
				(1171,30)
				(1187,24)
				(1223,48)
				(1231,46)
				(1259,46)
				(1279,32)
				(1283,52)
				(1291,54)
				(1303,40)
				(1307,48)
				(1319,42)
				(1327,26)
				(1367,42)
				(1399,40)
				(1423,36)
				(1427,46)
				(1439,48)
				(1447,52)
				(1451,52)
				(1459,62)
				(1471,40)
				(1483,50)
				(1487,40)
				(1499,50)
				(1511,58)
				(1523,44)
				(1531,60)
				(1543,32)
				(1559,66)
				(1567,32)
				(1571,68)
				(1579,42)
				(1583,64)
				(1607,62)
				(1619,60)
				(1627,52)
				(1663,62)
				(1667,40)
				(1699,52)
				(1723,50)
				(1747,46)
				(1759,32)
				(1783,58)
				(1787,60)
				(1811,62)
				(1823,52)
				(1831,38)
				(1847,42)
				(1867,54)
				(1871,38)
				(1879,58)
				(1907,34)
				(1931,48)
				(1951,42)
				(1979,54)
				(1987,52)
				(1999,46)
				(2003,54)
				(2011,50)
				(2027,48)
				(2039,42)
				(2063,56)
				(2083,56)
				(2087,58)
				(2099,54)
				(2111,72)
				(2131,38)
				(2143,34)
				(2179,54)
				(2203,66)
				(2207,70)
				(2239,62)
				(2243,64)
				(2251,48)
				(2267,72)
				(2287,36)
				(2311,70)
				(2339,66)
				(2347,52)
				(2351,52)
				(2371,54)
				(2383,50)
				(2399,52)
				(2411,72)
				(2423,48)
				(2447,58)
				(2459,70)
				(2467,50)
			};
		\end{axis}
	\end{tikzpicture}
\end{center}

This plot gives evidence for Rosa's conjecture about the number of perfect $1$-factorisations of $K_{n+1}$ when $n \equiv 3 \bmod 4$ is prime. It also suggests that the conjecture could be resolved in this case by looking at Latin squares $\mathcal{L}[a, a^{-1}]$ and potentially using \tref{t:linkmatcyc}.

The problem stated by Falconer also suggests a direction for future research. To solve Falconer's problem, it sufficed to construct a row-Hamiltonian Latin square that also had the property that no column permutation was a single cycle. This suggests exploring the following problem. Given relatively prime integers $a$ and $b$, can we construct a Latin square $L$ such that every row cycle of $L$ has length $a$ and every column cycle of $L$ has length $b$?

\section*{Acknowledgements}

The authors are grateful to Michael Kinyon for drawing their attention to Falconer's open problem. The second author also acknowledges Darryn Bryant, Barbara Maenhaut and Victor Scharaschkin for their work with him on an earlier bid to prove \tref{t:RCnotS} by an entirely different method.

\printbibliography

\end{document}